\documentclass[reqno,10pt]{amsart} 
\usepackage[dvipsnames,usenames]{color} 
\usepackage[toc,page]{appendix}

\usepackage{tikz} 
\usepackage{mathcomp,wasysym}  
 
\usepackage{hyperref}
\usepackage[mathscr]{euscript}   
       
\usepackage{cite}        
\usepackage{graphicx}  
\usepackage[all]{xy} \xyoption{arc} \xyoption{color} 

\usepackage{amsmath} 
\usepackage{amsthm} 
\usepackage{amsfonts}  
\usepackage{amssymb} 

\usepackage{verbatim}

\numberwithin{equation}{section} 

\newtheorem{corollary}{\sc Corollary}[section]
\newtheorem{theorem}{\sc Theorem}[section]
\newtheorem{lemma}{\sc Lemma}[section]

\newtheorem{remark}{\sc Remark}
\newtheorem{definition}{\sc Definition}[section]

\def\div{\operatorname{div}}

\def\n{\nonumber}

\def\nn{{\scriptstyle \mathcal{N} }}
\def\nnn{{\scriptscriptstyle \mathcal{N} }}
\def\tt{{\scriptstyle \mathcal{T} }}
\def\ttt{{\scriptscriptstyle \mathcal{T} }}
\def\ss{{\scriptstyle \mathcal{S} }}
\def\sss{{\scriptscriptstyle \mathcal{S} }}
\def\dvt{{ \delta v \cdot \tau}}
\def\dvp{{ \delta v' \cdot \tau}}
\def\dt{ \nabla_\ttt}
\def\pp{{\scriptstyle \mathscr{P} }}
\def\q{{\scriptstyle \mathscr{Q} }}
\def\x{{\scriptstyle \mathscr{X} }}
\def\y{ {\scriptstyle \mathscr{Y} }}
\def\X{\mathscr{X}}
\def\Y{\mathscr{Y}}
\def\Z{\mathscr{Z}}
\def\P{\mathscr{P}}
\def\eo{\mathscr{E}_1}
\def\eo{{\mathrm{e}_1}}
\def\et{\mathscr{E}_2}
\def\et{{\mathrm{e}_2}}

\def\p{\partial}

\def\Op{{\Omega^+}}
\def\Om{{\Omega^-}}

\def\D{{\mathcal D}}
\def\bdy #1{\partial #1}
\def\cls #1{\overline {#1}}

\def\um{{u^-}}
\def\umo{{u_1^-}}
\def\umt{{u_2^-}}
\def\dn{\boldsymbol{ \delta} \eta }
\def\du{\boldsymbol{ \delta} u^-}
\def\duo{\boldsymbol{ \delta} u^-_1}
\def\dut{\boldsymbol{ \delta} u^-_2}

\def\oz{\overline{z}}
\def\uz{\underline{z}}

\def\vr{\vec{r}}
\def\vl{\vec{l}}
\def\M{\mathscr M}
\def\aleph{\mathscr K}

\title[No splash singularities for  vortex sheets]
{On the impossibility of finite-time splash singularities for vortex sheets}

\author[D. Coutand]{Daniel Coutand}
\address{CANPDE, Maxwell Institute for Mathematical Sciences and department of Mathematics,
Heriot-Watt University, Edinburgh, EH14 4AS, UK}
\email{D.Coutand@ma.hw.ac.uk}

\author[S. Shkoller]{Steve Shkoller}
\address{Mathematical Institute,
University of Oxford,
Andrew Wiles Building,
Radcliffe Observatory Quarter,
Woodstock Road,
Oxford, OX2 6GG, UK}
\email{shkoller@maths.ox.ac.uk}

\date{June 30, 2015}

\begin{document}

\begin{abstract} In fluid dynamics,
an interface {\it splash} singularity occurs when a locally smooth interface self-intersects in finite time.  By means of elementary arguments, we prove that such a 
singularity cannot occur in finite time for vortex sheet evolution, i.e. for the two-phase incompressible Euler equations.    
We prove this by contradiction;  we assume that a splash singularity does indeed occur in finite time.  Based on this assumption, we find precise
 blow-up rates for the components of the
velocity gradient which, in turn, allow us to characterize
the geometry of the evolving interface just prior to self-intersection.  The constraints on the geometry then lead to an impossible outcome, showing that our assumption of a finite-time splash singularity was false.
\end{abstract}

\maketitle

\tableofcontents


\section{Introduction}
\subsection{The interface splash singularity}
The fluid interface  {\it splash singularity}   was introduced by  Castro, C\'{o}rdoba,  Fefferman, Gancedo, \& G\'{o}mez-Serrano in \cite{CaCoFeGaGo2013}.   
A {\it splash singularity} occurs when a fluid interface remains locally smooth but self-intersects in finite time.     For the two-dimensional water waves problem,  Castro, C\'{o}rdoba,  Fefferman, Gancedo, \& G\'{o}mez-Serrano
\cite{CaCoFeGaGo2013} showed that a splash singularity occurs in finite time using methods from complex analysis together with a clever transformation of the equations.
In  Coutand \& Shkoller \cite{CoSh2014}, we   showed the existence of a finite-time splash singularity for the water waves equations in two or three-dimensions (and, more generally, for the one-phase Euler equations), using a very different 
approach, founded  upon an approximation of the self-intersecting fluid domain by a sequence of smooth fluid domains, each with non  self-intersecting
boundary.

\subsection{The two-fluid incompressible Euler equations}
A natural question, then, is whether a splash singularity can occur for vortex sheet evolution, in which two phases of the fluid are present.
   Consider the two-phase incompressible Euler equations:
Let  $\mathcal{D} \subseteq \mathbb{R}^2  $ denote an open, bounded set,
which comprises the volume occupied by two incompressible and inviscid fluids with
different densities. At the initial time $t=0$,   we let $\Op$ denote the volume occupied by the {\it lower}
fluid with density $\rho^+$ and we let $\Om$ denote the volume occupied by
the {\it upper} fluid with density $\rho^-$.   Mathematically, the sets
$\Op$ and $\Om$ denote two disjoint open bounded
subsets of $\D$ such that $\cls\D = \cls{\Op} \cup \cls{\Om}$ and $\Op \cap \Om = \emptyset$.
The {\it material interface} at time $t=0$ is given by
$\Gamma := \cls\Op \cap \cls\Om$, and
$\bdy\D = \bdy (\Om \cup \Op) / \Gamma$.   (We can also consider the case that $\Op =  \mathbb{T}  \times (-1,0)$, $\Om =  \mathbb{T}  \times (0,1)$,
and $\Gamma = \mathbb{T}  \times \{0\}$.)

For time $t \in [0,T]$ for some $T>0$ fixed, 
 $\Op(t)$ and $\Om(t)$ denote the time-dependent volumes of the two fluids, respectively, separated
by the moving material interface $\Gamma(t)$.
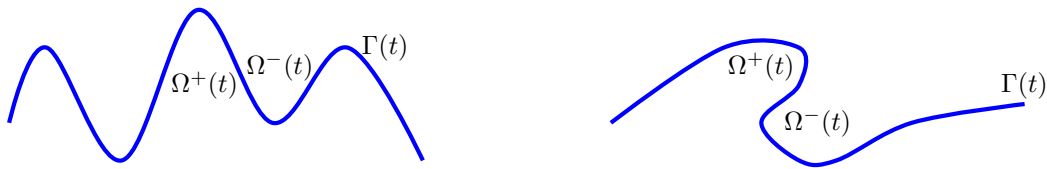
\begin{figure}[h]
\begin{tikzpicture}[scale=0.5]
    \draw (10,2) node { $\Gamma(t)$}; 
    \draw (7.2,1.5) node { $\Omega^-(t)$}; 
    \draw (5.2,1.) node { $\Omega^+(t)$}; 
    \draw[color=blue,ultra thick] plot[smooth,tension=.6] coordinates{( 0,0) (1,2) (3,-1) (5, 3) (7, 0) (9, 2) (11,-1) };
    
     \draw (27,1) node { $\Gamma(t)$}; 
    \draw (20,1.5) node { $\Omega^+(t)$}; 
    \draw (21.5,0) node { $\Omega^-(t)$}; 
    \draw[color=blue,ultra thick] plot[smooth,tension=.6] coordinates{( 16,0) (19,2) (21,2) (21,1)   (20, 0) (21, -1) (22,-1) (24,0) (27, .5)};
\end{tikzpicture} 
\caption{Two examples of the evolution of a vortex sheet $\Gamma(t)$  by the Euler equations.   The two fluid regions are denoted by
$\Omega^+(t)$ and $\Omega^-(t)$.} \label{fig1}
\end{figure}
 Let
$u^\pm$
and $p^\pm$ denote the velocity field and pressure function,
respectively, in $\Omega^\pm(t)$.
A  planar vortex sheet $\Gamma(t)$ evolves according to the incompressible and irrotational Euler equations:
\begin{subequations}\label{euler}
\begin{alignat}{2}
\rho^\pm(u^\pm_t + u^\pm\cdot D u^\pm) + D p^\pm &=- \rho^\pm g \et &&\qquad
\text{in \ \
$\Omega^\pm(t)$}\,, \\
\operatorname{curl} u^\pm=0 , \ \ \ \div u^\pm &= 0 &&\qquad \text{in \ \
$\Omega^\pm(t)$}\,, \\
p^+-p^- &= \sigma H &&\qquad \text{on \ \ $\Gamma(t)$}\,,\\
(u^+ - u^-)\cdot \nn  &=0 &&\qquad \text{on \ \ $\Gamma(t)$}\,, \\
u^-\cdot N  &= 0 &&\qquad\text{on \ \ $\bdy\D$}\,,\label{nonslipbc}\\
u(0) &= u_0 &&\qquad \text{on \ \ $\{t=0\}\times \D$}\,,\\
\mathcal{V} (\Gamma(t)) & = u^+(t) \cdot \nn(t) && \,,  \label{normalspeed}
\end{alignat}
\end{subequations}
where $ \mathcal{V} (\Gamma(t))$ denotes the speed of the moving
interface $\Gamma(t)$ in the normal direction, and $\nn(\cdot ,t)$ denotes the
outward-pointing unit normal to $\Gamma(t)$ (pointing into $\Omega^-(t)$), $N$ denotes the outward-pointing unit normal to the fixed boundary $\partial \D$, $g$ denotes gravity, and $\et$ is the vertical unit vector $(0,1)$.  Equation (\ref{normalspeed}) indicates
that 
$\Gamma(t)$ moves with the normal component of the
fluid velocity.  The variables $0 < \rho^\pm$
denote the densities of the
two fluids occupying $\Omega^\pm(t)$, respectively,  $H(t)$ is
twice the mean curvature of $\Gamma(t)$, and $\sigma >0$ is the surface
tension parameter which we will henceforth set to one.   For notational simplicity, we will also set $\rho^+=1$ and $\rho^-=1$.    

Via an elementary proof by contradiction, we prove that a finite-time splash singularity cannot occur for vortex sheets governed by (\ref{euler}).   We rule-out a single splash
singularity in which one self-intersection occurs, as well as the case that many (finite or infinite) simultaneous self-intersections occur. We also rule out a {\it splat}
singularity, wherein the interface $\Gamma(t)$ self-intersects along a curve (see \cite{CaCoFeGaGo2013} and \cite{CoSh2014} for a precise definition).

\subsection{Outline of the paper}
In Section \ref{sec_Lagrangian}, we introduce Lagrangian coordinates (using the flow of $u^-$) for the purpose of fixing the domain and the material interface.  Rather than using an arbitrary parameterization of the evolving interface $\Gamma(t)$, we specifically use the Lagrangian parameterization which has some important features  for our analysis that general parameterizations do not.  With this parameterization defined,
we 
state the main theorem of the paper in Section \ref{sec_mainresult} which states that a finite-time splash singularity cannot occur in this setting.  In Section \ref{sec_evolution}, we derive the evolution equations for the vorticity along the
interface as well as the evolution equation for the tangential derivative of the vorticity; the latter plays a fundamental role in our analysis.
In particular,  under the assumption that the tangential derivative of vorticity blows-up in finite time, we find the precise blow-up rates for
the components of $ \nabla \um(\cdot ,t)$.   Letting $\eta( \cdot , t): \Gamma \to \Gamma(t)$ denote the Lagrangian parameterization of the vortex sheet,
and supposing that the two reference points $x_0$ and $x_1$ in $\Gamma$ evolve toward one another so that $|\eta(x_0,t) - \eta(x_1,t)| \to 0$
as $t \to T$, 
in  Section \ref{sec_geometry}, we find the evolution equation for the distance $ \dn (t) = \eta(x_0,t) - \eta(x_1,t)$ between the two contact points.  We
can determine that the two portions of the curve $\Gamma(t)$ converge towards self-intersection in an essentially horizontal approach.

Finally, using the evolution equation for $\dn(t)$, we prove our main theorem
in Section \ref{sec_proof}; in particular, we show that our assumption of a finite-time self-intersection of the curve $\Gamma(t)$ as $t \to T$ leads to
the following contradiction:  we first show that $u^-_1(\eta(x_0,T),T) - u^-_1(\eta(x_1,T),T) =0$, where $\eo$ is the tangent vector at 
$\eta(x_0,T)$, and then we proceed to show that 
$u^-_1(\eta(x_0,T),T) - u^-_1(\eta(x_1,T),T) \neq 0$.   We first arrive at this contradiction for a single splash singularity, meaning that one self-intersection
point exists for $\Gamma(T)$; then, we proceed to prove that a finite (or even infinite) number of self-intersections also cannot occur.  We conclude
by showing that a {\it splat} singularity, wherein $\Gamma(T)$ self-intersects along a curve rather than a point, also cannot occur.

\subsection{A brief history of prior results}

\subsubsection{Local-in-time well-posedness} We begin with a short history of the local-in-time existence theory for the free-boundary
incompressible Euler equations.
For the {\it irrotational} case of the water waves problem, and for
2-D fluids (and hence 1-D interfaces), the earliest local existence results
were obtained by Nalimov \cite{Na1974}, Yosihara \cite{Yo1982}, and 
Craig \cite{Cr1985} for initial data near equilibrium.  Beale, Hou, \&
Lowengrub \cite{BeHoLo1993} proved that the linearization  of the 2-D water wave 
problem is well-posed if the Rayleigh-Taylor sign condition
$\frac{\p p}{\p n} < 0$ on $\Gamma \times \{t=0\} $
is satisfied by the initial data (see \cite{Ra1878} and \cite{Taylor1950}).
Wu \cite{Wu1997} established local well-posedness for the 2-D 
water waves problem and showed that, due to irrotationality, the Taylor sign condition is satisfied.  
Later Ambrose \& Masmoudi \cite{AmMa2005}, proved local
well-posedness of the 2-D water waves problem as the limit of zero surface tension.   Disconzi \& Ebin \cite{DiEb2014,DiEb2015} have considered the limit of surface tension tending to infinity.
For 3-D fluids (and  \mbox{2-D} interfaces), Wu \cite{Wu1999} used Clifford analysis to prove local existence of the 3-D water waves problem with {\it infinite depth}, again showing that the Rayleigh-Taylor sign 
condition is always satisfied in the irrotational case by virtue of the maximum
principle holding for the potential flow.  Lannes \cite{La2005} provided a proof
for the {\it finite depth case with varying bottom}.     Recently, Alazard, Burq \& Zuily \cite{AlBuZu2012} have established low regularity solutions (below
the Sobolev embedding) for the
water waves equations.   

 The first local well-posedness result for
the 3-D incompressible Euler equations without the irrotationality assumption was obtained by 
Lindblad \cite{Li2004}  in the case that the domain is diffeomorphic to the unit ball  using a Nash-Moser iteration.
Coutand \& Shkoller \cite{CoSh2007} proved  local well-posedness  for arbitrary initial geometries that have at least $H^3$-class
boundaries  without derivative loss.     Shatah \& Zeng \cite{ShZe2008} established a priori estimates
for this problem using an infinite-dimensional geometric formulation, and  Zhang \& Zhang proved well-poseness by extending
the complex-analytic method of Wu \cite{Wu1999} to allow for vorticity.  Again, in the latter case the domain was with infinite depth.

\subsubsection{Long-time existence}  It is of great interest to understand if solutions to the
Euler equations can be extended for all time when the data is sufficiently smooth and small, or if a finite-time
singularity can be predicted for other types of initial conditions.  

 Because of irrotationality,  the water waves problem does not suffer from vorticity concentration;  therefore, singularity formation
 involves only the loss of regularity of the interface or interface collision.   In the case that the irrotational fluid is infinite in the horizontal directions,
 certain dispersive-type properties can be made use of.
 For sufficiently smooth and small data,
 Alvarez-Samaniego \& Lannes \cite{AlLa2008} proved existence of solutions to the
water waves problem  on large time-intervals (larger than predicted by energy estimates), 
and provided a rigorous justification for a variety of  asymptotic regimes.   By constructing a transformation to remove
the quadratic nonlinearity, combined with decay estimates for the linearized problem (on the infinite half-space domain), Wu \cite{Wu2009} established  an almost global existence result (existence on time intervals which are exponential in the size of the data) for the 2-D water
waves problem with sufficiently small data.     In a different framework, Alazard, Burq \& Zuily \cite{AlBuZu2012} have also proven this result.
  Using position-velocity potential holomorphic coordinates, Hunter, Ifrim, \& Tataru \cite{HuIfTa2014} have also
proved almost global existence of the 2-D water waves problem.

  Wu \cite{Wu2011}  proved global existence in 3-D for small data.   Using
the method of spacetime resonances,  Germain, Masmoudi, and Shatah \cite{GeMaSh2009} also established global existence
for the 3-D irrotational problem for sufficiently small data.     More recently, global existence for the 2-D water waves problem with small data was
established by Ionescu \& Pusateri \cite{IoPu2014},  Alazard \& Delort \cite{AlDe2013a,AlDe2013b}, and Ifrim \& Tataru \cite{IfTa2014a,IfTa2014b}.

\subsubsection{The finite-time splash  and splat singularity}  The finite-time {\it splash} and {\it splat singularities} were introduced by
Castro, C\'{o}rdoba, Fefferman, Gancedo, \& G\'{o}mez-Serrano \cite{CaCoFeGaGo2013}; therein, using methods from complex analysis,
they proved that a locally smooth interface can self-intersect in finite time  for the 2-D water waves equations and hence established the existence
of finite-time splash and splat singularites (see also  \cite{CaCoFeGaLo2011a}
and \cite{CaCoFeGaLo2011b}).    In Coutand \& Shkoller \cite{CoSh2014}, we established the  existence of finite-time splash and splat singularities for the 2-D and  3-D water waves
and Euler equations (with vorticity) using an approximation of the self-intersecting domain  by a sequence of standard Sobolev-class domains, each with
non self-intersecting boundary.    Our approach 
 can be applied to many  one-phase hyperbolic free-boundary problems, and shows that splash singularities can occur with surface tension, with
 compressibility, with magnetic fields, and for many one-phase hyperbolic free-boundary problems.

Recently, Fefferman, Ionescu,  \& Lie \cite{FeIoLi2013} have proven that a splash singularity cannot occur for planar vortex sheets (or two-fluid interfaces) with surface tension.  Their proof relies on a 
sophisticated harmonic analysis of the integral kernel of the Birkhoff-Rott equation.    Other than vortex sheet evolution for the two-phase Euler equations, it is of interest to determine the possibility of
finite-time splash singularities for other fluids models.   In this regard, 
 Gancedo \& Strain \cite{GaSt2013} have recently shown that a  finite-time splash singularity cannot occur for the three-phase Muskat equations.   In 
addition to the study of other fluids models, it is also of great interest to determine a mechanism for the loss of regularity of the evolving interface, 
which, in turn, could allow for finite-time self-intersection.

\section{Fixing the fluid domains using the Lagrangian flow of $u-$}\label{sec_Lagrangian}

Let $\tilde\eta$ denote the Lagrangian flow map of $u^-$ in $\Omega^-$ so that 
$\tilde\eta_t(x,t) = u^-(\tilde\eta(x,t),t) $ for $x\in \Omega^-$ and $t \in (0,T)$, with initial condition 
$\tilde\eta(x,0)=x$.   Since $ \operatorname{div} u^-=0$, it follows that $ \det \nabla\tilde \eta=1$.  By a theorem of \cite{DaMo1990}, we define $ \Psi: \Omega^+ \to \Omega^+(t)$ as incompressible extension of $\tilde \eta$, satisfying
$\det \nabla \Psi =1$ and $\|\Psi\|_{H^s(\Omega^+)} \le C \|\eta^-|_\Gamma\|_{H^{s-1/2}(\Gamma)}$ for $s >2$.  We then set 
$$
\eta(x,t)= \left\{
\begin{array}{cc}
\tilde \eta(x,t), & x \in \overline{\Omega^-} \\
\Psi(x,t), & x \in \Omega^+
\end{array}
\right. .
$$
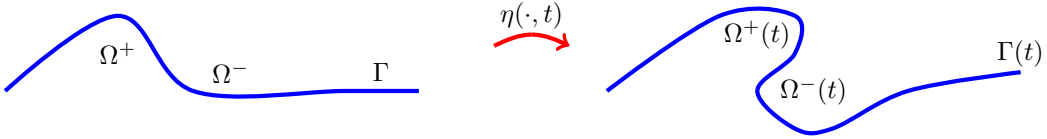
\begin{figure}[h]
\begin{tikzpicture}[scale=0.5]
    \draw (10,.5) node { $\Gamma$}; 
    \draw (3,1) node { $\Omega^+$}; 
    \draw (6,.5) node { $\Omega^-$}; 
    \draw[color=blue,ultra thick] plot[smooth,tension=.6] coordinates{( 0,0)  (3,2) ( (5, 0) (9, 0) (11,0) };
    
    \draw (14,2) node { $\eta( \cdot ,t)$}; 
      \draw[ultra thick, red]  (13,1.2) sin (14,1.5);  
      \draw[ultra thick, red] [->] (14,1.5) cos (15,1.2);  

     \draw (27,1) node { $\Gamma(t)$}; 
    \draw (20,1.5) node { $\Omega^+(t)$}; 
    \draw (21.5,0) node { $\Omega^-(t)$}; 
    \draw[color=blue,ultra thick] plot[smooth,tension=.6] coordinates{( 16,0) (19,2) (21,2) (21,1)   (20, 0) (21, -1) (22,-1) (24,0) (27, .5)};
\end{tikzpicture} 
\caption{The mapping $\eta( \cdot , t)$ fixes the two fluid domains and the interface.    The moving interface $\Gamma(t)$ is the image of $\Gamma$ by $\eta( \cdot ,t)$.} \label{fig2}
\end{figure}

We define the following quantities set on the fixed domains and boundary:
\begin{alignat*}{2}
v^\pm&= u^\pm \circ \eta\,, &&\qquad \text{in \ \ $\Omega^\pm \times [0,T]$}\,, \\
q^\pm &= p^\pm \circ \eta \,, &&\qquad \text{in \ \ $\Omega^\pm \times [0,T]$}\,, \\
A&= [ \nabla \eta ]^{-1} \,,&&\qquad \text{in \ \ $\mathcal{D}  \times [0,T]$}\,, \\
\mathcal{H} &= H \circ \eta \,, &&\qquad \text{on \ \ $\Gamma \times [0,T]$}\,, \\
\delta v &= v^+ - v^-  \,, &&\qquad \text{on \ \ $\Gamma \times [0,T]$}\,, 
\end{alignat*}

The  momentum equations (\ref{euler}a) can then be written on the fixed domains $\Omega^\pm$ as 
\begin{subequations}\label{leuler}
\begin{alignat}{2}
v^+_t +  \nabla v ^+\, A\, (v^+-\Psi_t ) + A^T \nabla q^+&= -g \et &&\qquad
\text{in \ \
$\Omega^+ \times [0,T]$}\,, \\
v^-_t + A^T \nabla q^- &= -g \et &&\qquad
\text{in \ \
$\Omega^- \times [0,T]$}\,,
\end{alignat}
\end{subequations}
and the pressure jump condition (\ref{euler}c) is
$ \delta q = \mathcal{H} $ on $\Gamma \times [0,T]$, where $ \delta q = q^+ -q^-$.

Using  the Einstein summation convention, $[ \nabla v^+ \, A\, (v^+-\Psi_t )]^i = {v^+}^i,_r A^r_j(v^+_j-\p_t\Psi_j)$.    This is the advection term; when
$\Psi$ is the identity map, we recover the Eulerian description, while if $\Psi$ is the Lagrangian flow map, then we recover the Lagrangian description.  The form
(\ref{leuler}.a) is called the Arbitrary Lagrangian Eulerian (ALE) description of the fluid flow in $\Omega^+$

\section{The main result} \label{sec_mainresult}
In \cite{ChCoSh2008, ChCoSh2010}, we proved that if at time $t=0$,  $u_0^\pm \in H^k(\Omega^\pm)$ and  $\Gamma$ of class $H^{k+1}$ for integers $k\ge 3$,
then there exists a solution $(u^\pm (\cdot , t) , \Gamma(t))$ of the system (\ref{euler}) satisfying
$u^\pm \in L^ \infty (0,T_0; H^k(\Omega^\pm(t)))$ with $ \Gamma(t)$ being of class $H^{k+1}$, for all $t\in [0,T_0]$, for some $T_0>0$.   (See also \cite{ShZe2011} and \cite{Pusateri2011}.)

\begin{theorem}[No finite-time splash singularity]\label{thm1} 
Let $\D$ be a bounded domain of class $H^4$. We assume the existence of a closed curve $\Gamma\subset\D$ of class 
$W^{4, \infty }$ which does not self-intersect and such that $\D=\Omega^+\cup\Gamma\cup\Omega^-$, where the open sets 
$\Omega^+$ and $\Omega^-$ are connected and disjoint and do not intersect $\Gamma$. 
Our assumption of non self-intersection means that $\Omega^+$ and $\Omega^-$ are both (locally) on one side of $\Gamma$.

Let $u^\pm$ be a solution to (\ref{euler}) on $[0,T)$ such that $u^\pm \in H^3(\Omega^\pm(t))$ and $ \Gamma(t)$ is of class $W^{4,\infty }$ 
for each $t \in [0,T)$.   Suppose that
\begin{enumerate}
\item  $\Omega^+(t)$ and $\Omega^-(t)$ are both (locally) on one side of $\Gamma(t)$ for all $t\in [0,T)$;
\item  there exists a  constant   $0 <\mathcal{M}<\infty $, such that
$$\text{ for all  } \ t\in [0,T)\,,  \ \ \operatorname{dist} (\Gamma(t),\partial\D)>{\frac{1}{\mathcal{M} }} \,.$$
and
\begin{equation}\label{bounds}
\sup_{t\in [0,T)} \left( \left\|u^+( \cdot , t) \right\|_{W^{2, \infty }(\Gamma(t))}    + \left\| H ( \cdot , t) \right\|_{W^{2, \infty }(\Gamma(t))}  \right) <  \mathcal{M} \,.
\end{equation} 
\end{enumerate}
Then $\Gamma(t)$ 
cannot self-intersect \textcolor{black}{at time $t=T$}; that is, there does not exist a finite-time splash singularity.
  \end{theorem} 
Note, that we give a precise definition for the $W^{k, \infty }(\Gamma(t))$-norm below in Definition \ref{def_W2infi}.
\begin{remark} The condition (1) in  Theorem \ref{thm1}, requiring   $\Omega^+(t)$ and $\Omega^-(t)$ to  both (locally) be on one side of 
$\Gamma(t)$ for all $t\in [0,T)$, is equivalent to requiring the chord-arc function to be \textcolor{black}{strictly positive for all $t$ in $[0,T)$ (without specifying a lower bound as $t\rightarrow T$, other than $0$)}. 
 \end{remark}

\begin{remark} 
In Theorem \ref{thm1}, we have assumed that $\D=\Omega^+(t)\cup\Gamma(t)\cup\Omega^-(t)$ is a bounded domain simply because the local well-posedness
theorem for the two-phase Euler equations given in \cite{ChCoSh2008} used such a geometry; however, as our proof by contradiction relies on a local analysis in
a spacetime region near an assumed point (or points) of self-intersection of the curve $\Gamma(t)$, we can also treat the case that our two fluids occupy all of $ \mathbb{R}^2  $ or occupy
a channel geometry with periodic boundary conditions in the horizontal direction.

As part of condition (2) in  Theorem \ref{thm1} for the case that $\D$ is bounded, we assume that 
$\operatorname{dist} (\Gamma(t),\partial\D)>{\frac{1}{\mathcal{M} }} $
so that the moving interface $\Gamma(t)$ stays away from the fixed domain boundary $\partial \D$.
\end{remark}

\section{Evolution equations on $\Gamma$ for the vorticity and its tangential derivative}\label{sec_evolution}

\subsection{Geometric quantities  defined on $\Gamma$ and $\Gamma(t)$} \label{subsec::chart} We set
\begin{alignat}{2}
\nn(x,t)&= \text{ unit normal vector field on $\Gamma(t)$}\,, \qquad && n = \nn \circ \eta \n\\
\tt(x,t)&= \text{ unit tangent vector field on $\Gamma(t)$}\,,  && \tau = \tt \circ \eta \n\,.
\end{alignat} 
We choose the unit-normal  $\nn( \cdot ,t)$ to point into $\Omega^-(t)$.    In a sufficiently small neighborhood $ \mathcal{U} $ of the material interface 
$\Gamma$ at $t=0$,  
we choose a local chart $\theta: B(0,1) \to \mathcal{U} $.  The unit ball $B(0,1)$ has coordinates $(x_1,x_2)$, and 
$\theta : \{ (x_1,x_2) \ : \ x_2=0\} \to \mathcal{U} \cap \Gamma$,
$\theta  \{ (x_1,x_2) \ : \ x_2>0\} \to \mathcal{U} \cap \Omega^-$,  and 
$\theta  \{ (x_1,x_2) \ : \ x_2<0\} \to \mathcal{U} \cap \Omega^+$. In order to define a tangent vector, we also assume that the length $|\theta'(x_1,0)|$ of the vector $\theta'(x_1,0)$ is bounded away from $0$ by some constant $C>0$. For notational convenience in our computations, we shall write $\eta \circ \theta$
simply as $\eta$.
We define
$$ G(x,t)=|\eta'(x,t)|^{-1} \,,\text{ where } 
( \cdot )' = \partial ( \cdot ) / \partial x_1\,.$$     
Hence, 
\begin{equation}\label{tau_def}
\tau( x, t) = G \eta'(x,t)\,, \  n( x, t) =G \eta'^\perp(x,t)\,, \  x^\perp=(-x_2,x_1)\,.
\end{equation} 

On $\Gamma(t)$, we  let $ \nabla _\ttt$ denote the tangential derivative, i.e., the derivative in the direction of the unit tangent vector $\tt$.
Let $f$ denote any Eulerian quantity. Then, by the chain-rule, 
\begin{equation}\label{tan_der}
(\nabla _\ttt f ) \circ \eta = G (f \circ \eta )' \,.
\end{equation} 

\begin{definition}[$W^{k, \infty }(\Gamma(t))$-norm] \label{def_W2infi}
For a function $f(\cdot,t) : \Gamma(t) \to \mathbb{R}  $ and integers $k \ge 0$, we  define
$$\|f(\cdot,t)\|_{W^{k,\infty}(\Gamma(t))}= \sum_{i=0} ^k\| \nabla^i _\ttt f(\cdot,t)\|_{L^{\infty}(\Gamma(t))}
\,.$$
\end{definition} 

\begin{remark} From our assumed bounds (\ref{bounds}) we have that $|\nabla_{\ttt} u^+|\le\mathcal{M}$. 
Since $ \operatorname{div} u^+=0$, we have  that $|\nabla_{\nnn} u^+\cdot \nn|=|\nabla_{\ttt} u^+\cdot\tt|\le\mathcal{M}$, 
and since $ \operatorname{curl} u^+=0$,  $|\nabla_{\nnn} u^+\cdot \tt|=|-\nabla_{\ttt} u^+\cdot\nn|\le\mathcal{M}$, which shows that 
\begin{equation}
\label{120814.1}
\|\nabla u^+\|_{L^{\infty}(\Gamma(t))}\le\mathcal{M}\,
\end{equation} (where the norm of a matrix is chosen to be the maximum of the absolute value of all four components).
\end{remark}

\begin{remark}\label{remark2}
We now define $\phi$ to be the flow map of $u^+$ in $\overline{\Omega^+}$. With the chart $\theta$ introduced above, and
with $x=(x_1,0)$,
we then infer from $\phi_t(\theta(x),t)=u^+(\phi(\theta(x),t),t)$ that
$$[\phi_t(\theta(x),t)]'=\nabla u^+(\phi(\theta(x),t),t) \ [\phi(\theta(x),t)]'    \,.
$$
Therefore, 
\begin{equation*}
\frac{d}{dt} \left|[\phi(\theta(x),t)]'\right|^2=2 [\phi(\theta(x),t)]' \cdot \left( \nabla u^+(\phi(\theta(x),t),t) \ [\phi(\theta(x),t)]' \right)
\ge -4 \mathcal{M}|[\phi(\theta(x),t)]'|^2\,,
\end{equation*}
where the inequality follows from  (\ref{120814.1}). Thus
\begin{equation}
\label{120814.2}
|[\phi(\theta(x),t)]'|^2 \ge e^{-4\mathcal{M} t} |\theta'(x)|^2\ge e^{-4\mathcal{M} t} C^2>0 \,.
\end{equation}
 Therefore, the unit tangent vector to $\Gamma(t)$ can be defined simply as
\begin{equation*}
\tt(\phi(\theta(x),t))=\frac{[\phi(\theta(x),t)]'}{|[\phi(\theta(x),t)]'|}\,,
\end{equation*}
or with our notational convention of  writing $ \phi \circ \theta$ simply as $\phi$, 
\begin{equation*}
\tt(\phi)=\frac{\phi'}{|\phi'|}\,.
\end{equation*}
\end{remark}

\begin{remark} Using the same argument as in Remark \ref{remark2},
if $\|\nabla u^-( \cdot ,t) \|_{L^\infty(\Omega(t))}$ is bounded from above (which is the case for $t<T$ for a solution $u^-\in L^\infty (0,T;H^3(\Omega^-(t)))$ so long as
there is no self-intersection of $\Gamma(t)$),  then the flow map $\eta$ of $u^-$ satisfies an identity similar to (\ref{120814.2}), ensuring that
the definition of $G(x,t)$ is well-defined for all $t \in [0,T)$.
\end{remark}

\subsection{Evolution equation for the vorticity on $\Gamma$} 

Equation (\ref{leuler}a) is $v^+_t +  \nabla v ^+\, A\, (v^+-\Psi_t ) + A^T \nabla q^+= -g \et $.  By definition, on  $\Gamma$,
$\Psi_t = v^-$, so that $v^+-\Psi_t = \delta v$.  Since $ \delta v \cdot n=0$ on $\Gamma$, we see that $ \delta v = ( \delta v \cdot \tau) \tau$.
Hence, the advection term can be written (using the Einstein summation convention) as 
$\frac{ \p v^+}{\p x_r} A^r_j \tau_j ( \delta v \cdot \tau)$.
From (\ref{tau_def}), $ \tau_j = G \eta'_j$ which in our local coordinate system is the same as $G \frac{\p \eta_j}{\p x_1}$.   Since $A = [ \nabla \eta ]^{-1} $,
we see that $A^r_j \  \frac{\p \eta_j}{\p x_1} = \delta ^r_1$, where $ \delta ^r_1$ denotes the Kronecker delta.

It follows that  on $\Gamma$, (\ref{leuler}a) takes the form
\begin{equation}\label{lleuler}
v^+_t +G {v^+}' \, \delta v \cdot \tau + A^T \nabla q^+=-g \et \,.
\end{equation} 

Equation  (\ref{leuler}b) does not have the advection term, and remains the same on $\Gamma$.
Subtracting  (\ref{leuler}b)  from (\ref{lleuler}a), taking the scalar product of this difference with $ \tau $,  and using that
$ \delta q=\mathcal{H}$,
yields
$$
\delta v_t \cdot \tau + G {v^+}' \cdot \tau ( \dvt) + G \mathcal{H}'=0 \,,
$$
from which it follows that
\begin{equation}\label{aom3}
(\dvt)_t + G {v^+}' \cdot \tau ( \dvt) + G \mathcal{H}'=0
\text{ on } \Gamma \times [0,T) \,,
\end{equation} 
where we have used the fact that $\tau_t = G (v' \cdot n )n$ and $ \delta v\cdot n=0$.  Using (\ref{tan_der}), we write (\ref{aom3}) as
\begin{equation}\label{jump_v}
(\dvt)_t + [ \nabla _\ttt u^+ \cdot \tt \circ \eta ] ( \dvt) +\nabla _\ttt H \circ \eta=0
\text{ on } \Gamma \times [0,T) \,.
\end{equation} 

\subsection{Evolution equation for derivative of  vorticity $\dt \delta u \cdot \tt$} On $\Gamma$, 
we denote the tangential derivative by $\dt$.   The chain-rule (\ref{tan_der})  shows that the tangential derivative of vorticity along particle 
trajectories can be written as
\begin{equation}\label{eq11}
[\dt \delta u \cdot \tt] \circ \eta = G \delta v'\cdot \tau  \,.
\end{equation} 
    Our analysis will rely on the evolution
equation for $G \delta v'\cdot \tau $.   By differentiating (\ref{jump_v}), we find that
\begin{equation}\label{eq1}
(\dvp)_t +  [G {v^+}' \cdot \tau]  (\dvp) + (\dvt)[G {v^+}'\cdot \tau ]' + (G \mathcal{H} ')' =0 \,.
\end{equation} 
Defining our ``forcing function'' $ \mathcal{A} $  to be 
\begin{align}
\mathcal{A} & = (\dvt)G[G {v^+}'\cdot \tau ]' + G(G \mathcal{H} ')' \nonumber \\
&=  (\dvt) \nabla _\ttt (\nabla _\ttt u^+ \cdot \tt) \circ \eta  + \nabla _\ttt (\nabla _\ttt H) \circ \eta \,,
\label{bigA}
\end{align} 
we see that equation (\ref{eq1}) is simply
\begin{equation}\label{ns3}
( \dvp)_t + G {v^+}' \cdot \tau ( \dvp) + G ^{-1} \mathcal{A} =0 \,.
\end{equation} 
Multiplying (\ref{ns3}) by $G$ and 
commuting $G$ with the time-derivative shows that
\begin{equation}\n
( G\dvp)_t + G({v^-}' \cdot \tau+ {v^+}' \cdot \tau) ( G\dvp) + \mathcal{A} =0 \,.
\end{equation} 
Writing ${v^-}'\cdot \tau = -\dvp + {v^+}' \cdot \tau $, we arrive at the desired evolution equation
\begin{equation}\label{lvorticity}
(G \dvp)_t - (G \dvp)^2 + 2 G {v^+}' \cdot \tau (G\dvp) + \mathcal{A} =0 \,.
\end{equation} 
Notice that the coefficient $2 G {v^+}' \cdot \tau = 2 \nabla _\ttt u^+\cdot \tt \circ \eta$, as well as the forcing function $ \mathcal{A} $, 
are both bounded as a consequence of our assumed bounds (\ref{bounds}) on $u^+$ and the parameterization of $z( \cdot , t)$ of $\Gamma(t)$.

\begin{remark}
In  \cite{FeIoLi2013},  Fefferman,  Ionescu, \& Lie use the notation $z( \alpha ,t )$ to denote a smooth parameterization of $\Gamma(t)$. 
In our analysis,  we will make use of
the Lagrangian parameterization $\eta(x,t)$ of $\Gamma(t)$ for points $x $ in the reference curve $\Gamma$.   
Our notation $\eta'$ corresponds to $\partial_\alpha z$ in
\cite{FeIoLi2013}.   Furthermore, our $\delta v \cdot \tau$ is the same as $ \frac{ \omega}{ | \partial_\alpha z|} $ in \cite{FeIoLi2013}.   The tangential
derivative of vorticity $[\dt \delta u \cdot \tt] \circ \eta $ corresponds to $ \partial_ \alpha \left(  \frac{ \omega}{ | \partial_\alpha z|}\right) /  | \partial_\alpha z|$
in \cite{FeIoLi2013}.
 \end{remark} 

\section{Bounds for $ \nabla u^-$ and the rate of blow-up}\label{sec_bounds}
\label{estimates}

\begin{lemma} \label{lemma1} Assuming (\ref{bounds}), 
\begin{equation}
\label{ns5}
\sup_{t \in [0,T]} \|v^-(\cdot,t)\|_{W^{1,\infty}(\Gamma)}\lesssim \mathcal{M}  \,.
\end{equation}
\end{lemma} 
\begin{proof}  
With $\tau_0=\tau(x,0)$,
solving  (\ref{jump_v}) using an integrating factor, we find that
\begin{align}
\delta v \cdot\tau=& \delta u_0 \cdot\tau_0\  \exp\left(-\int_0^t G{v^+}' \cdot \tau \right)
 -  \exp\left(-\int_0^t G{v^+}'\cdot\tau \right) \int_0^t  \nabla _\ttt H \circ \eta\  \exp\left(\int_0^s G{v^+}'\cdot\tau\right)\ ds\,. \label{zzt100}
\end{align}

We set $ \mathcal{I} (t) = \exp\left(\int_0^t \|G{v^+}'\cdot\tau\|_{L^\infty(\Gamma)}\right) $.  Since
$ G  {v^+}'\cdot \tau  =[\dt  u^+ \cdot \tt] \circ \eta$, by (\ref{bounds}), $ \mathcal{I} (t)$ is bounded.   It follows from (\ref{zzt100}) that
$$
\| \delta v \cdot\tau(\cdot , t)\|_{L^\infty(\Gamma)}
\le  \mathcal{I} (t)  \|\delta u_0\|_{L^\infty(\Gamma)}+   \mathcal{I} (t) \int_0^t \| \nabla _\ttt H\circ \eta \|_{L^\infty(\Gamma)}\,.
$$
Again from (\ref{bounds}), the tangential derivative of the mean curvature $ \nabla _\ttt H  \in W^{1, \infty }(\Gamma)$ so we see that
$\| \delta v \cdot\tau(\cdot , t)\|_{L^\infty(\Gamma)}$ is bounded.

Next, as $ \delta v \cdot n=0$, and $v^+ \cdot n$ is bounded according to (\ref{bounds}), we find that
$ \|v^-(\cdot,t)\|_{L^{\infty}(\Gamma)}\lesssim \mathcal{M} $ for all $t \in [0,T]$.
Then, from (\ref{ns3}),
\begin{align*}
\delta v' \cdot\tau=& \delta u_0' \cdot\tau_0\  \exp\left(-\int_0^t G{v^+}' \cdot \tau \right)
 -  \exp\left(-\int_0^t G{v^+}'\cdot\tau \right) \int_0^t  G ^{-1} \mathcal{A} \exp\left(\int_0^s G{v^+}'\cdot\tau\right)\ ds\,.
\end{align*}
so that with $G ^{-1} = | \eta '|$, 
\begin{align} 
\| \delta v' \cdot\tau(\cdot , t)\|_{L^\infty(\Gamma)}
& \le  \mathcal{I} (t)  \|\delta u_0' \cdot \tau_0\|_{L^\infty(\Gamma)}
+   \mathcal{I} (t) \int_0^t \|  | \eta '( \cdot , s)|  \mathcal{A}(\cdot ,s) \|_{L^\infty(\Gamma)} ds \,. \label{aom4}
\end{align} 
From the fundamental theorem of calculus, 
\begin{align*}
|\eta'( \cdot ,s) | & \le | \eta'( \cdot ,0)| + \int_0^s |{v^-}' (\cdot , r)| dr \\
& \le | \eta'( \cdot ,0)| + \int_0^s |{v^+}' (\cdot , r)| dr  + \int_0^s | \delta v'(\cdot , r)| dr \\
& \le \mathcal{M}   + \int_0^s | \delta v'(\cdot , r) \cdot n( \cdot ,r)| dr  + \int_0^s | \delta v'(\cdot , r) \cdot \tau( \cdot ,r)| dr \,,
 \end{align*} 
where we have used our assumed bounds (\ref{bounds}) for the last inequality.  Next, since $ \delta v \cdot n=0$ on $\Gamma$, we see that
 $ \delta v' \cdot n = - \delta v\cdot n'$; as $n' =  (H \circ \eta) \tau$, and as $\| H \circ \eta ( \cdot , t)\|_{L^\infty(\Gamma)}$ and 
 $\| \delta v \cdot\tau(\cdot , t)\|_{L^\infty(\Gamma)}$ are bounded,  we see from (\ref{aom4}) that
 $$
 \| \delta v' \cdot\tau(\cdot , t)\|_{L^\infty(\Gamma)} \lesssim \mathcal{M} + T \mathcal{M} \int_0^t  \| \delta v' \cdot\tau(\cdot , s)\|_{L^\infty(\Gamma)} ds\,.
 $$
 By taking the convention that $\lesssim$ incorporates $T$ (which we view in this paper as a given constant, namely the eventual finite-time of self-intersection), this shows that
 $$
 \| \delta v' \cdot\tau(\cdot , t)\|_{L^\infty(\Gamma)} \lesssim \mathcal{M} + \mathcal{M} \int_0^t  \| \delta v' \cdot\tau(\cdot , s)\|_{L^\infty(\Gamma)} ds\,.
 $$ 
Hence, by Gronwall's inequality, $\sup_{t\in [0,T]}  \| \delta v' \cdot\tau(\cdot , t)\|_{L^\infty(\Gamma)} $ is bounded.   We have already shown
that $\sup_{t\in [0,T]}  \| \delta v' \cdot n(\cdot , t)\|_{L^\infty(\Gamma)} $ is bounded; thus, $\sup_{t\in [0,T]}  \| \delta v' \|_{L^\infty(\Gamma)} $ is
bounded,
from which we may conclude that 
$ \|{v^-}' (\cdot , t)  \|_{L^{\infty}(\Gamma)}\lesssim \mathcal{M} $ for all $t \in [0,T]$.   
\end{proof} 

\begin{remark} Note that $u^-$ is Lipschitz continuous, uniformly on any time interval $[0,t]$ with $t < T$.   This, in turn, allows us to define the
Lagrangian flow map $\eta$ in a classical sense for any time interval $[0,t]$ for $t < T$.  We then extend this definition of $\eta$ to the time interval $[0,T]$ by $\eta(x,T) = x+ \int_0^T v^- (x,s)ds$ by the bounds in Lemma \ref{lemma1}.
\end{remark} 

\begin{lemma} \label{lemma2} Assuming (\ref{bounds}), 
\begin{equation}
\label{ns10}
\sup_{t \in [0,T]} \| \nabla u^-(\cdot,t)\|_{L^\infty(\eta(\Omega^-,t))}\lesssim \frac{\mathcal{M} }{\min_{\Gamma}|\eta'(\cdot,t)|} \,.
\end{equation}
\end{lemma} 
\begin{proof} 
From (\ref{eq11}) and Lemma \ref{lemma1}, 
$
\left\|[\dt \delta u \cdot \tt] \circ \eta\right\|_{L^\infty(\Gamma)} \lesssim { \mathcal{M} }/{\min_{\Gamma}|\eta'( \cdot ,t)|}$.
Then, we see that
$\max_{y\in \eta(\Gamma,t)}  \left|\dt \delta u \cdot \tt \right| \lesssim{ \mathcal{M} }/{\min_{\Gamma}|\eta'( \cdot ,t)|}$.
Hence, with our assumed bounds (\ref{bounds}),
\begin{equation}\label{good1}
\max_{y\in \eta(\Gamma,t)}  \left|\dt  u^- \cdot \tt \right|\ \lesssim\frac{ \mathcal{M} }{\min_{\Gamma}|\eta'( \cdot ,t)|}\,.
\end{equation}

Next, as $ \delta u\cdot \nn=0$ (where recall that $\delta u = u^+-u^-$ on $\Gamma(t))$, we have the identity
$0= \nabla_{\ttt} (\delta u\cdot \nn)=(\nabla_{\ttt} \delta u)\cdot \nn  + \delta u \cdot  \nabla_{\ttt}\nn$; hence, we see that
\begin{equation*}
\nabla_\ttt u^-\cdot \nn=\nabla_\ttt u^+\cdot \nn+\delta u \cdot  \nabla_{\ttt}\nn\,.
\end{equation*}
Lemma \ref{lemma1} provides us with  $L^\infty(\Gamma)$ control of $u^-$; hence, with (\ref{bounds}),  it follows that
\begin{equation}
\label{ns18.e}
\max_{y\in \eta(\Gamma,t)}  \left|[\nabla_{\ttt}  u^-\cdot\nn ](y)\right| \lesssim \mathcal{M} \,.
\end{equation}
The inequalities (\ref{good1}) and (\ref{ns18.e}) together with the fact that $ \operatorname{div} u^- = \operatorname{curl} u^-=0$ in $\eta(\Omega^-,t)$ implies that
for any $t<T$,
\begin{equation}
\label{ns8}
\| \nabla u^-(\cdot,t)\|_{L^\infty(\eta(\Gamma,t))} \lesssim\frac{ \mathcal{M} }{\min_{\Gamma}|\eta'( \cdot ,t)|}\,.
\end{equation}
As $\Delta \nabla  u^-=0$ in $\eta(\Omega^-,t)$, the maximum and minimum principle applied to each component of $ \nabla u^-$, together with  (\ref{ns8}), 
provide the inequality (\ref{ns10}).
\end{proof} 

\begin{remark} 
As a consequence of Lemma \ref{lemma2}, we see that  $ \sup_{y \in \Gamma(t)} \| \nabla u^-(y,t)\|_{L^\infty(\eta(\Omega^-,t))} \to \infty $ as
$t \to T$  if and only if $\lim _{ t \to T} |\eta'(x,t)| \to 0$ for some $ x \in \Gamma$.   If we assume that there 
are distinct  points $x_0,x_1 \in \Gamma$ 
which come into contact, such that $\eta(x_0,T) = \eta(x_1,T)$  and that such an intersection point is unique at time $t=T$, then 
$| \nabla \um ( \cdot ,t)|$ can only blow-up at the contact point $\eta(x_0,T) $.

The explanation is as follows:  since $\um$ is harmonic, by using a smooth cut-off function $\varphi$ whose support does not intersect $\eta(x_0,T)$, and proceeding as in the proof of (\ref{052715.1}) (just after (\ref{eps_def2})), 
elliptic estimates 
show that $| \nabla \um ( x,t)|$ must be bounded for $x \in 
\operatorname{spt}(\varphi)$, namely away from $x_0$.

Next, suppose that  $| \nabla \um ( \eta(x_0,t),t)| $ remains bounded as $t \to T$; 
then, by employing a similar argument as we used to establish (\ref{120814.2}) (considering now the flow $\eta$ of $u^-$), we obtain that
$|\eta'(x_0,t)| \ge \lambda >0$ as $t \to T$ for some constant $\lambda$.  
By continuity of $\eta'$, this means that $|\eta'(x,t)|  >0$ in a small neighborhood of $\eta(x_0,t)$
which means that, by Lemma \ref{lemma2}, $| \nabla \um ( x,t)|$ cannot blow-up as $t\to T$ for $ x$ close to $x_0$.

\end{remark}

\begin{theorem} \label{thm_blowup} With the assumed bounds (\ref{bounds}), if
 there is a sequence $t_n \to T$ such that
\begin{equation}\label{assump1} \max_{x \in \Gamma}\left| [\nabla_{\ttt} \delta u\cdot\tt](\eta(x,t_n),t_n) \right| \to \infty \,,\end{equation} 
 then for $ 0 <\epsilon \ll 1$, there exists $t_0(\textcolor{black}{\epsilon})$ such that \textcolor{black}{ $T-t_0( \epsilon ) < \epsilon $} and
\begin{equation}
\label{ns18.f}
\max_{y\in\eta(\overline{\Omega^-},t)} |\nabla  u^-(y,t)| \le \frac{1+ \epsilon }{T-t}\ \ \ \ \forall t\in [t_0(\textcolor{black}{\epsilon}),T)\,.
\end{equation}
Furthermore, if there exists a unique point of $\Gamma(T)$ such that there are two distinct points $x_0,x_1 \in \Gamma$ with
 $\eta(x_0,T)=\eta(x_1,T)$ with tangent vector to $\Gamma(T)$ at $\eta(x_0,T)$ given by
$\eo$, then
\begin{equation}
\label{u12epsilon}
\max_{y\in\eta(\overline{\Omega^-},t)} \left| \frac{\p u_2^-}{\p x_1} (y,t) \right| \le \frac{\epsilon}{T-t}\ \ \ \ \forall t\in [t_0(\textcolor{black}{\epsilon}),T)\,.
\end{equation}
\end{theorem} 
\begin{remark} 
We note that $0< \epsilon \ll 1$ is a {\it fixed} positive constant which only depends on the initial data and the bound $ \mathcal{M}$  in (\ref{bounds}).  Note also that $t_0( \epsilon) $ depends on $ \epsilon $, and will be chosen closer and closer to $T$ in the course of the
proof, and is eventually fixed as a function of $ \epsilon$.
\end{remark} 
\begin{proof}  
{\it Step 1. Blow-up rate for the derivative of vorticity $[\nabla_{\ttt} \delta u\cdot\tt](\eta(x_0,t),t)$  as $t \to T$.}  
We first  suppose that for some $x_0 \in \Gamma$, $ \left| [\nabla_{\ttt} \delta u\cdot\tt](\eta(x_0,t_n),t_n) \right| \to \infty$, and 
establish that  $[\nabla_{\ttt} \delta u\cdot\tt](\eta(x_0,t),t)$ (which, recall, equals $G\dvp(x_0,t)$) has a precise blow-up rate under the assumption (\ref{assump1}).

We set 
$$\mathscr{X} (x_0,t) = G\dvp(x_0,t)\,,$$
and define the coefficient function
$$ \mathfrak{A}(x_0,t) = 2 G {v^+}' \cdot \tau (x_0,t) \,.$$
Then,
 (\ref{lvorticity}) reads 
\begin{equation}\label{ns17}
 \mathscr{X}_t (x_0,t)  - \mathscr{X}^2 (x_0,t) + \mathfrak{A}(x_0,t) \, \mathscr{X} (x_0,t) = - \mathcal{A}( x_0, t ) \,,
 \end{equation} 
 where $ \mathcal{A} (x,t)$ is defined in (\ref{bigA}).
 This equation can be written as 
$$\left[ \exp\int_0^t \mathfrak{A} (x_0,s)ds\ \mathscr{X} (x_0,t) \right]_t - \exp\int_0^t \mathfrak{A} (x_0,s)ds\ \mathscr{X}^2 (x_0,t)=- \exp\int_0^t \mathfrak{A} (x_0,s)ds\ \mathcal{A}(x_0,t) 
$$
so that
\begin{align}
\int_0^t \exp\left(\int_0^s \mathfrak{A} (x_0,r)dr\right) \mathscr{X}^2 (x_0,s) ds&=
\exp\left(\int_0^t \mathfrak{A} (x_0,s)ds\right) \mathscr{X} (x_0,t) - \mathscr{X} (x_0,0) \nonumber\\
&\ \ +\int_0^t \exp\left(\int_0^s \mathfrak{A} (x_0,r)dr\right) \mathcal{A}(x_0,s) ds \label{eq12}
\,.
\end{align} 
Thanks to (\ref{bounds}), $ \mathfrak{A} (x_0,t)$ has a minimum and maximum on $[0,T]$.  Hence, there are positive constants $c_1,c_2,c_3$ such
that for any $t \in [0,T)$,
$$
c_1\int_0^t \mathscr{X}^2(x_0,s)ds - c_3 \le \mathscr{X} (x_0,t)\le c_2\int_0^t \mathscr{X}^2 (x_0,s) ds + c_3 \,,
$$ 
and by (\ref{assump1}),  the limit as $t\rightarrow T$ is well-defined and 
\begin{equation}
\label{ns18}
\lim_{t\rightarrow T} \mathscr{X} (x_0,t)=\infty\,,
\end{equation}

For $t>\bar t_0$ sufficiently close  to $T$, we can then divide (\ref{ns17}) by $\mathscr{X}^2$, and integrate from $\bar t_0$ to $t$, to find that
\begin{equation*}
-\frac{1}{\mathscr{X} (x_0,t)} +\frac{1}{\mathscr{X} (x_0,\bar t_0)}-t+\bar t_0+ \int_{t_0}^t \left(\frac{ \mathfrak{A} (x_0,s)}{ \mathscr{X} (x_0,s)} + \frac{ \mathcal{A} (x_0,s)}{\mathscr{X}^2(x_0,s)}\right) ds =0\,.
\end{equation*}
Using the limit in (\ref{ns18}), 
\begin{equation}\label{aom}
\frac{1}{\mathscr{X} (x_0,\bar t_0)}-T+\bar t_0+ \int_{t_0}^T \left(\frac{ \mathfrak{A} (x_0,s)}{ \mathscr{X} (x_0,s)} + \frac{ \mathcal{A} (x_0,s)}{\mathscr{X}^2 (x_0,s)}\right) ds =0\,,
\end{equation} 
from which we obtain  the following identity: for $t \in [t_0,T)$,
\begin{equation}
\label{ns18bis}
\mathscr{X} (x_0,t) = \left[   T-t - \int_t^T \left(\frac{ \mathfrak{A} (x_0,s)}{ \mathscr{X} (x_0,s)} + \frac{ \mathcal{A} (x_0,s)}{\mathscr{X}^2 (x_0,s)}\right) ds    \right]^{-1} \,,
\end{equation}
since we can replace $t_0$ with $t$ in (\ref{aom}).

From (\ref{ns18}), this formula implies that the integrand is small as $t$ is close to $T$, and then provides the rate of blow-up:
\begin{equation}
\n
\lim_{t\rightarrow T}  \mathscr{X} (x_0,t) ({T-t})=1\,.
\end{equation}
Using (\ref{bounds}), we see that
\begin{equation}
\lim_{t\rightarrow T} 
[\nabla_{\ttt} \um \cdot\tt](\eta(x_0,t),t)\ ({T-t})=-1 \,.
\label{ns19}
\end{equation} 

\noindent
{\it Step 2. Maximum of vorticity derivative blows-up on $\Gamma(t)$.}  Having established the blow-up rate for $[\nabla_{\ttt} \delta u\cdot\tt](\eta(x_0,t),t)$, 
we shall next prove that for any $t\in [0,T)$, the quantity $\max_{ x \in \Gamma}[\nabla_{\ttt} \delta u\cdot\tt](\eta(x,t),t )$ (which equals $ \max_{ x \in \Gamma}G\dvp(x,t)$)  has the same blow-up rate.
For each $x \in \Gamma$ and $t\in [0,T)$, we set
\begin{equation}\label{Acoeff}
 \mathfrak{A}(x,t) = 2 G {v^+}' \cdot \tau (x,t)  \text{ and } \mathscr{X} (x,t) = G\dvp(x,t)\,.
\end{equation} 
 Following (\ref{eq12}), we see that
 \begin{equation}\label{eq13}
 \mathscr{X}(x,t) \ge \exp\left(-\int_0^t \mathfrak{A} (x,s)ds\right) \mathscr{X}(x,0) -\exp\left(-\int_0^t \mathfrak{A} (x,s)ds\right)\int_0^t \exp\left(\int_0^s \mathfrak{A} (x,r)dr\right) \mathcal{A}(x,s) ds
\,;
\end{equation} 
 hence, there exists a positive constant $c_4$ such that $ \mathscr{X}(x,t) > -c_4$.  Since $\mathscr{X} _t = \mathscr{X}^2  - \mathfrak{A} \, \mathscr{X} - \mathcal{A}$, there is a positive constant $c_5$,
  $$\mathscr{X}_t> \mathscr{X}^2/2 -c_5\,.$$ 
It follows that if 
 $\mathscr{X}(x,t_0)\ge \sqrt{2c_5}$, then $\X(x, \cdot )$ is increasing on $[t_0,T)$.    
 For $x\in\Gamma$ we choose $t_0( \epsilon) <T$ sufficiently close to $T$ so that  for $0 < \epsilon  \ll 1$ fixed,
 \begin{equation}\label{ns18.c}
 \X(x,t_0( \epsilon )) > \sqrt{2c_5} +1 + \frac{8c_6}{ \epsilon }\,, \ \ \ \ c_6 = \sup_{(t,x) \in [0,T]\times\Gamma}\left( | \mathfrak{A} (x,t)| + \ \mathcal{A} (x,t)|\right) \,,
 \end{equation} 
 with $c_6$ denoting a bounded constant thanks to (\ref{bounds}).
Since $\X(x,\cdot )$ is increasing for such an $x$,  for $t\in [t_0(\epsilon ),T)$, the limit of $\X(x,t)$ as $t\rightarrow T$ is well-defined in the interval
 $(1+\sqrt{2c_5}+8 c_6/ \epsilon ,\infty]$, and
 thus so is the limit of $\frac{1}{\X(x,t)}$.   Analogous to  (\ref{ns18bis}), we obtain that
\begin{equation*}
{\X(x,t)} = \left[{\frac{1}{\lim_{t\rightarrow T} \X(x,t)}+T-t+ \int_{T}^t\left(\frac{ \mathfrak{A} (x,s)}{ \X(x,s)} + \frac{ \mathcal{A} (x,s)}{\X ^2(x,s)}\right)  ds}\right]^{-1} \,.
\end{equation*}
From (\ref{ns18.c}), we then have that for all $t\in [t_0( \epsilon ),T)$,
\begin{equation*}
{\X(x,t)} \le \left[{\frac{1}{\lim_{t\rightarrow T} \X(x,t)}+(T-t)(1- \epsilon )}\right]^{-1} \,
\end{equation*}
and since $\lim_{t\rightarrow T} \X(x,t)\ge 0$, then for all $t< T$,
\begin{equation}\label{eq14}
{\X(x,t)} \le \frac{1}{(T-t)(1- \epsilon )}\,.
\end{equation}

\vspace{.1 in}
\noindent
{\it Step 3. Blow-up rate for $\nabla \um$ in $\overline{\Omega^-(t)}$ as $t\to T$.} From (\ref{eq14}), for  any $t\in[t_0( \epsilon ),T)$, 
\begin{equation}
\label{ns18.d}
\max_{y\in \eta(\Gamma,t)} \left|[\nabla_{\ttt} \delta u\cdot\tt ](y,t)\right| \le \frac{1+ 2 \epsilon }{(T-t)}\,.
\end{equation}

The inequalities (\ref{ns18.e}) and (\ref{ns18.d}), together with the fact that $ \operatorname{div} u^- = \operatorname{curl} u^-=0$ in $\eta(\Omega^-,t)$, show that
\begin{equation}\label{eq15}
\max_{y\in\eta(\Gamma,t)} | \nabla  u^-(y,t)| \le \frac{1+ 2 \epsilon }{T-t}\,,
\end{equation}
where $\max_{y\in\eta(\Gamma,t)} | \nabla  u^-(y,t)|$ denotes the maximum over all of the components of the matrix $ \nabla \um$.
Now, for any fixed $t\in [0,T)$, since each component of  $\nabla u^-$ is harmonic in the domain $\eta(\Omega^-,t)$, the maximum and minimum principles
together with  the boundary estimate (\ref{eq15}) shows that (\ref{ns18.f}) holds.

\vspace{.1 in} 
\noindent
{\it Step 4. Asymptotic estimates for the components of $ \nabla \um$ as $t \to T$ in an $ \epsilon $-neighborhood of the splash.}
Since
\begin{equation*}
\frac{\p u^-}{\p x_1}:= \nabla _\eo u^-  =(\tt\cdot \eo )\nabla_\ttt u^- + (\nn\cdot \eo ) \nabla_\nnn u^-\,,
\end{equation*}
we  have that
\begin{align} 
\frac{\p u_2^-}{\p x_1}
&=(\tt\cdot \eo) \nabla_\ttt u^-\cdot (\tt\cdot \et\ \tt+\nn\cdot \et\ \nn)+
(\nn\cdot \eo) \nabla_\nnn u^-\cdot (\tt\cdot \et\ \tt+\nn\cdot \et\ \nn) \n\\
&=(\tt\cdot \eo)(\tt\cdot\et) \nabla_\ttt u^-\cdot  \tt         +(\tt\cdot \eo)(\nn \cdot \et) \nabla_\ttt u^-\cdot\nn+(\tt\cdot \et)(\nn \cdot \eo) \nabla_\nnn u^-\cdot \tt\n\\
&\  +
(\nn\cdot \eo)(\nn\cdot\et) \nabla_\nnn u^-\cdot \nn\,.
\label{ns18.g}
\end{align} 
 
By rotating our coordinate system, if necessary, we suppose that 
 the tangent and normal directions to $\Gamma(T)$ at  $\eta(x_0,T)$ are given by the standard basis vectors $\eo=(1,0)$ and $\et=(0,1)$, respectively \textcolor{black}{(which we refer to as the horizontal and vertical directions, respectively)}.  

   Next,
choose a point $\eta(x,t) \in \Gamma(t)$ in a small neighborhood of $\eta(x_0,t)$, and  let the curve $\ss(t)$ denote that portion of $\Gamma(t)$ that
connects $\eta(x_0,t)$ to $\eta(x,t)$.   Let  $\vl(t): [0,1] \to \ss(t)$ denote a unit-speed parameterization such that $\vl(t)(1) = \eta(x,t)$ and $\vl(t)(0) = \eta(x_0,t)$.
Then,
\begin{align}
  \nn(\eta(x,t),t) \cdot \eo -  \nn(\eta(x_0,t),t) \cdot \eo &= \int_{\sss(t)} \nabla ( \nn \cdot \eo ) \cdot d\vl \nonumber\\
  \tt(\eta(x,t),t) \cdot \et -  \tt(\eta(x_0,t),t) \cdot \et & = \int_{\sss(t)} \nabla ( \tt \cdot \et ) \cdot d\vl \,.\label{052015.1}
\end{align}

From our assumed bounds (\ref{bounds}), there is a constant $c_7>0$ such that for $t \le T$
\begin{align} 
& |\nn(\eta(x,t),t) \cdot \eo -  \nn(\eta(x_0,t),t) \cdot \eo | +  |\tt(\eta(x,t),t) \cdot \et - \tt(\eta(x_0,t),t) \cdot \et   | 
\nonumber\\
& \qquad\qquad\qquad\qquad\qquad  \le c_7 |\eta(x,t) - \eta(x_0,t)| \,. \label{cs-c7}
\end{align}

Next, with $G= |\eta'| ^{-1} $, 
we compute that
\begin{align*} 
 \tau _t &  = (G {v^-}'\cdot n) \, n 
  = - (G \delta v' \cdot n) \, n + (G {v^+}'\cdot n) \, n \\
  & =  (G \, n' \cdot  \delta v) \, n + (G {v^+}'\cdot n) \, n 
   = \left[(\nabla _\ttt \nn\cdot \delta u) \, \nn  + (\nabla _\ttt u^+\cdot \nn) \, \nn\right] \circ \eta \,,
\end{align*} 
where we have used (\ref{tan_der}) in the last equality.   There is a similar formula for $n_t = - (G {v^-}'\cdot n) \, \tau$. 
It follows from Lemma \ref{lemma1} and our assumed bounds (\ref{bounds}) that  
\begin{equation}\label{aom5}
\sup_{t\in[0,T]} \left( \|\tau_t(\cdot ,t)\|_{L^ \infty (\Gamma)} +  \|n_t(\cdot ,t)\|_{L^ \infty (\Gamma)} \right) \lesssim \mathcal{M} \,.
\end{equation} 
Then, using the fundamental theorem of calculus, we see that
\begin{alignat*}{2}
 \nn(\eta(x_0,t),t) \cdot \eo &=  \nn(\eta(x_0,t),t) \cdot \eo -  \nn(\eta(x_0,T),T) \cdot \eo &&= \int_T^t \partial_tn(x_0,s)\cdot \eo ds \\
 \tt(\eta(x_0,t),t) \cdot \et &=  \tt(\eta(x_0,t),t) \cdot \et -  \tt(\eta(x_0,T),T) \cdot \et &&= \int_T^t \partial_t  \tau  (x_0,s)\cdot \et ds
\end{alignat*} 
so that (by readjusting the constant $c_7$ if necessary), we have that
\begin{equation} 
\label{052015.1bis}
 |\nn(\eta(x_0,t),t) \cdot \eo| +  |\tt(\eta(x_0,t),t) \cdot \et|  \le c_7 (T-t) \,.
\end{equation} 
Next,  we
\begin{equation}
\label{eps_def} 
\begin{array}{l}
\text{choose $t_0(\textcolor{black}{\epsilon}) \in [0,T)$  and a sufficiently small neighborhood $\gamma_0(\textcolor{black}{\epsilon}) \subset \Gamma$ 
of $x_0$ s.t.}  \\
\left\{
\begin{array}{l}
(T-t) < \min\left( \frac{\epsilon }{100c_7\textcolor{black}{(1+\mathcal{M})}}, \epsilon \right)  \ \text{ and } \  
 |\eta(x,t) - \eta(x_0,t)| < {\frac{\epsilon }{2c_7}}  \\
|\nn(\eta(x,t),t) \cdot \eo| +  |\tt(\eta(x,t),t) \cdot \et| < \epsilon  
\end{array}\right\}    \ \forall \ x \in \gamma_0(\textcolor{black}{\epsilon}), \ t\in [t_0(\textcolor{black}{\epsilon}),T)  \,,
\end{array} 
\end{equation} 
where the constant $c_7$ was defined in (\ref{cs-c7})
%
%
%
Consequently,  from (\ref{ns18.e}), (\ref{eq15}) and (\ref{ns18.g}), we see that
\begin{equation*}
\left|\frac{\p u_2^-}{\p x_1}(\eta(x,t),t)\right|\le \frac{ 3 \epsilon}{T-t} + |\nabla_\ttt u^-\cdot\nn|(\eta(x,t),t) \textcolor{black}{+2\epsilon}| \nabla_\nnn u^-\cdot \tt|(\eta(x,t),t)\,,
\end{equation*} 
which thanks to (\ref{ns18.e}) and the fact that $ \operatorname{curl} u^-= \nabla _\ttt u^- \cdot \nn - \nabla _\nnn u^-\cdot \tt=0$, provides us with
\begin{equation*}
\left|\frac{\p u_2^-}{\p x_1}(\eta(x,t),t)\right|\le \frac{ 3 \epsilon}{T-t} + c_8\mathcal{M} \ \ \forall \ x \in \gamma_0(\textcolor{black}{\epsilon}), \ t\in [t_0(\textcolor{black}{\epsilon}),T)\,,
\end{equation*}
for a constant $c_8>0$.
Thus , by choosing $t_0(\epsilon)$ closer to $T$  if necessary, we have  that
\begin{equation}\label{jj0}
\left|\frac{\p u_2^-}{\p x_1}(\eta(x,t),t)\right|\le \frac{3 \epsilon}{T-t} \ \ \forall \ x \in \gamma_0(\textcolor{black}{\epsilon}), \ t\in [t_0(\textcolor{black}{\epsilon}),T)\,.
\end{equation}

In a similar fashion, we
\begin{equation}
\label{eps_def2} 
\begin{array}{l}
\text{choose $t_0(\textcolor{black}{\epsilon}) \in [0,T)$  and a sufficiently small neighborhood $\gamma_1(\textcolor{black}{\epsilon}) \subset \Gamma$ 
of $x_1$ s.t.}  \\
\left\{
\begin{array}{l}
(T-t) < \min\left( \frac{\epsilon }{100c_7\textcolor{black}{(1+\mathcal{M})}}, \epsilon  \right)  \ \text{ and } \  
 |\eta(x,t) - \eta(x_1,t)| < {\frac{\epsilon }{2c_7}}  \\
|\nn(\eta(x,t),t) \cdot \eo| +  |\tt(\eta(x,t),t) \cdot \et| < \epsilon  
\end{array}\right\}    \ \forall \ x \in \gamma_1(\textcolor{black}{\epsilon}), \ t\in [t_0(\textcolor{black}{\epsilon}),T) 
\end{array} 
\end{equation} 
and such that the  inequality (\ref{eps_def}) holds.  
 Now, we choose $x \in \Gamma$ but in the complement of $\gamma_0(\textcolor{black}{\epsilon}) \cup \gamma_1(\textcolor{black}{\epsilon})$.  For such
 an $x$, we have that
 $|\nabla u^-(\eta(x,t),t)| \le \mathcal{M} _ \epsilon < \infty $.  This bound is obtained as follows.

\def\bepst{\mathcal{B} _{\epsilon ,t}}
For each $t\in [t_0(\epsilon),T)$, we let $\bepst \subset \mathbb{R}^2  $ denote a small closed ball containing
$\eta(\gamma_0({\epsilon}) ,t) \cup \eta(\gamma_1({\epsilon}),t)$.
 \textcolor{black}{The ball $\bepst$ can be taken with a fixed radius independent of $t\in [t_0(\epsilon),T)$ (for $T-t_0( \epsilon )$ sufficiently small),
with a center which is simply translated as $t$ varies.  This is possible as we assume at that there is a single point of self-intersection
for the curve $\Gamma(T)$, and so the width of the domain $\Omega^-(t)$ cannot shrink to zero  in other locations as $t \to T$.}
With the unit tangent vector field $\tt$ defined on $\Gamma(t)$, we  define a smooth extension of $\tt$ to
the set $\Omega^-(t) \cap \bepst^c$, which is  possible since the interface $\Gamma(t)\cap \bepst^c$ remains $W^{4, \infty }$ for
all $t\in[0,T]$; we continue to denote this extension by $\tt$, and we note that the extension of $\tt$ does not necessarily 
have modulus $1$.
Since $\Gamma(t) \cap \bepst^c$ does not self-intersect for all $t\in[0,T]$ by the hypothesis (1) of Theorem \ref{thm1}, 
 there exists a minimum positive radius $r_ \epsilon >0$ such that for all
 $x \in \Gamma(t) \cap \bepst^c$ and all $t \in [t_0( \epsilon) , T]$, there exists a translated open ball $ \mathscr{B}_{\epsilon ,t,x}(r_ \epsilon )\subset\Omega^-(t)$ of radius $r_ \epsilon $  with
 $x \in \partial \mathscr{B}_{\epsilon ,t,x}(r_ \epsilon )$.  In other words, for each $x\in\Gamma(t)$ {\it away from the region of self-intersection}, there exists an open ball of smallest radius $r_ \epsilon$  that is contained in
 the set $\Omega^-(t)$ and such that $x$ is on the sphere of smallest radius.

We note that the radius $r_ \epsilon  \to 0$ as $ \epsilon \to 0$; hence, on 
 the domain $\Omega^-(t) \cap \bepst^c$,
we have an estimate of the type
\begin{equation}
\label{052915.1} 
\|\tt\|_{H^{3}(\Omega^-(t)\cap\bepst^c)}\lesssim C(\mathcal{M},\epsilon)\,,
\end{equation}
where $C(\mathcal{M},\epsilon)>0$ denotes a constant depending on $\mathcal{M}$ and $\epsilon$ (with 
$C( \mathcal{M} , \epsilon ) \to \infty $ as  $\epsilon\rightarrow 0$).

We now introduce the stream function $\psi^-$ such that $\um=\nabla^\perp \psi^-$; then, 
\begin{equation*}
\tt\cdot\nabla\psi^-=u^-\cdot\nn=u^+\cdot\nn\ \text{on}\ \Gamma(t)\,,
\end{equation*}
which then shows, using our bounds in (\ref{bounds}),  that
\begin{equation}
\label{052915.2}
\|\psi^-\|_{H^3(\Gamma(t))}\lesssim C(\mathcal{M}, \epsilon )\,.
\end{equation}
Furthermore, due to the conservation law
\begin{equation}\label{conservationlaw}
\frac{1}{2} \|\um(t)\|_{L^2(\Omega^-(t))}^2+\ \text{length of}\ \Gamma(t)=\frac{1}{2} \|u^-_0\|_{L^2(\Omega^-)}^2+\ \text{length of}\ \Gamma\,,
\end{equation}
we have that
\begin{equation}
\label{052915.3}
\|\psi^-\|_{H^1(\Omega^-(t))}\lesssim C(\mathcal{M}, \epsilon )\,,
\end{equation}
where we have used that $\|\psi^-\|_{H^1(\Omega^-(t))} \le C( \| \nabla \psi^-\|_{ L^2(\Omega^-(t)) }  + \| \psi^-\|_{ H^3(\Gamma(t))}  ) $
 and
(\ref{052915.2}).

Next, we fix $t\in [t_0(\epsilon),T]$, and choose 
 a smooth cut-off function $0\le \varphi(\cdot,t)\le 1$ whose support is contained in the complement of $\bepst$.
 Since $\Gamma(t) \cap \bepst^c$ is assumed to be of class $W^{4, \infty }$ for each $t\in [0,T]$,  we consider the following elliptic problem:
 \begin{subequations}
 \begin{alignat*}{2}
 \Delta  (\varphi \psi^-) & =  2 \nabla \varphi \cdot \nabla \psi^- + \Delta \varphi \ \psi^-
 \qquad  && \text{ in } \Omega^-_\epsilon(t)\,, \\
\varphi\psi ^- &= \varphi\psi^-  && \text{ on } \partial \Omega^-_\epsilon(t)\,, 
 \end{alignat*}
 \end{subequations}
where $\Omega^-_\epsilon(t)$ is a smooth open subset of $\Omega^-(t)$ containing $\Omega^-(t) \cap \bepst^c$,
 and where we have used the fact that $\psi^-$ is harmonic, since  $\operatorname{curl} u^-=0$.
From (\ref{052915.3}), (\ref{052915.2}), we have by elliptic regularity that 
\begin{equation}
\label{052915.4}
\|\varphi\psi^-\|_{H^2(\Omega^-_\epsilon(t))}\lesssim C(\mathcal{M},\epsilon)\,.
\end{equation}
We next consider the elliptic problem:
 \begin{subequations}
 \begin{alignat*}{2}
 \Delta  (\varphi \tt\cdot\nabla(\varphi\psi^-)) & =  2 \nabla (\varphi  \tt_i) \cdot \nabla \frac{\partial{(\varphi\psi^-)}}{\partial x_i} + \Delta (\varphi\tt_i) \ \frac{\partial{(\varphi\psi^-)}}{\partial x_i}   && \\
 & \qquad+  \varphi \tt \cdot \nabla \left[ 2 \nabla \varphi\cdot \nabla \psi^- + \Delta \varphi\, \psi^- \right]
 \qquad  && \text{ in } \Omega^\epsilon\,, \\
\varphi\tt\cdot\nabla(\varphi\psi^-) &= \varphi\tt\cdot\nabla (\varphi\psi^-) && \text{ on } \partial \Omega\epsilon\,, 
 \end{alignat*}
 \end{subequations}
Due to (\ref{052915.4}), (\ref{052915.2}) and (\ref{052915.1}), we have by elliptic regularity:
\begin{equation}
\label{052915.5}
\|\varphi\tt\cdot\nabla(\varphi\psi)\|_{H^2(\Omega^\epsilon)}\lesssim C(\mathcal{M},\epsilon)\,.
\end{equation}
In the same manner as we obtained (\ref{052915.5}) from (\ref{052915.4}), we can also obtain that
\begin{equation}
\label{052915.6}
\|\varphi\tt\cdot\nabla(\varphi\tt\cdot\nabla(\varphi\psi))\|_{H^2(\Omega^\epsilon)}\lesssim C(\mathcal{M},\epsilon)\,.
\end{equation}
By the trace theorem and the Sobolev embedding theorem, we infer from (\ref{052915.6}) that
\begin{equation*}
 \textcolor{black}{ \|\varphi^3   \nabla u^- \cdot\nn (\cdot,t)\|_{L^{\infty}(\Gamma(t))}}\lesssim C(\mathcal{M},\epsilon)\,.
  \end{equation*}

 Since $u^-$ is divergence and curl free this immediately ensures by the algebraic expression of the divergence and curl that
 \begin{equation}
\label{052715.1} 
  \|\varphi^3 \nabla  u^- (\cdot,t)\|_{L^\infty(\Gamma(t))}\lesssim  C( \mathcal{M}, \epsilon )\,,
  \end{equation}
  showing that $\nabla  u^-\cdot\nn(\eta(x,t),t)$ is bounded for $\eta(x,t)$ outside of $\bepst$. 
   Therefore, our previous estimates obtained for $x$ in $\gamma_0(\epsilon)$ and $\gamma_1(\epsilon)$ ensure that for all $x\in\Gamma$, $|\nabla u^-(\eta(x,t),t)| < \frac{\epsilon}{T-t}$ 
  for $t$ sufficiently close  to $T$;   thus,
 for $T-t_0(\textcolor{black}{\epsilon})$  sufficiently small (which means that once again, we have taken $t_0( \epsilon) $ even closer
 to $T$ if necessary),
\begin{equation*}
\max_{y\in\eta(\Gamma,t)} \left|\frac{\p u_2^-}{\p x_1}(y,t)\right|\le \frac{ 3\epsilon}{T-t}\ \ \ \ \forall t\in [t_0({\epsilon}),T)\,,
\end{equation*}
which, thanks to the maximum and minimum principles applied to the harmonic function $\frac{\p u_2^-}{\p x_1}$, provides us with (\ref{u12epsilon}).
Since $0< \epsilon\ll 1$, we replace $3 \epsilon $ by $ \epsilon $, and replace $1+2 \epsilon $ by $1+ \epsilon $.  This completes the proof.
\end{proof} 

\begin{corollary} \label{cor2}
 With  (\ref{bounds}) and  (\ref{assump1}) holding, 
 and for $ 0 <\epsilon \ll 1$,  there exists   $t_0( \epsilon ) \in [T- \epsilon , T) $ such that
\begin{equation}\label{improved}
\| \nabla_{\ttt} \um \cdot\tt ( \cdot , t) \|_{L^ \infty(\Gamma(t))} \le \frac{1+ (1+2c_6)(T-t)}{T-t} \ \ \forall t\in [{ t_0( \epsilon )},T) \,,
\end{equation} 
where the constant $c_6$ is defined in (\ref{ns18.c}).
\end{corollary} 
\begin{proof}  Using the notation from the proof of Theorem \ref{thm_blowup}, 
$$
 \X(x,t) =  \nabla_{\ttt} \delta u\cdot\tt ( \eta(x,t) , t) \,.
$$
and we recall that $\X(x_0,t) = \chi(t)$  and that  $\X(x,t)$ satisfies 
\begin{equation}\label{Xeq}
 \X_t(x,t)- \X ^2(x,t) + \mathfrak{A}(x,t )\, \X(x,t) = - \mathcal{A} (x,t) \,.
\end{equation} 

We let  $ \delta t = T-t$, 
and fix $0 < \epsilon \ll 1$.   
 Since $\lim_{t\rightarrow T}  \X(x_0,t) ({T-t})=1$,   for $  \delta t$ sufficiently small, we have that
 $$(1- \epsilon) \delta t ^{-1} \le \X(x_0,t) \le (1+ \epsilon ) \delta t^{-1} .$$
 Substituting this inequality into (\ref{ns18bis}), we see that
\begin{equation}\label{chi_est0}
\X(x_0,t) \le \frac{1}{ (1- c_6 \delta t) \delta t} \le \frac{ 1+ 2 c_6 \delta t}{ \delta t} \,.
\end{equation} 
If we replace $x_0$ with $x_1$, then (\ref{chi_est0}) continues to hold.

Now, for the sake of contradiction, we will assume that there exists a sequence of points $(x^*, t^*)$,  with $x^* \in \Gamma$ and $t^*$ converging to $T$, such that
\begin{equation}\label{assume10}
\X(x^*,t^*) > \frac{1+ (1+2c_6) (T-t^*)}{ T-t^*} \,.
\end{equation} 
We will later prove that set of possible contact points $ x_i \in \Gamma$,  such that $\eta(x_0,T) = \eta(x_1,T) = \eta(x_i,T)$,  is finite.
Then, from this set of all possible reference points which can self-intersect at time $t=T$, we relabel $x_0$ so that $x_0$ is the limit
of a subsequence of points $x^*$ converging toward it along $\Gamma$.   Henceforth, we restrict the sequence of points $(x^*, t^*)$ to the subsequence which converges to the point $x_0$.

\textcolor{black}{By Remark 7, if there exists $C>0$ such that $|\nabla u^-(\eta(x_0,t),t)|\le C$ for any $t\in [t_0(\epsilon), T)$, we would also have the existence of a neighborhood of $x_0$ on $\Gamma$ such that for any $x$ in this neighborhood, $|\nabla u^-(\eta(x,t),t)|\le 2 C$ for any $t\in [t_0(\epsilon), T)$. This would then make (\ref{assume10}) impossible. Therefore, we have $\X(x_0,T)\rightarrow\infty$ as $t\rightarrow T$.}

We assume that  this
point is $x_0$ (for otherwise we can reverse the labels on the two points $x_0$ and $x_1$).    Notice 
that since $x^* \to x_0$ as $t^* \to T$, then
for $T-t^*$ sufficiently small,
\begin{equation}\label{xx0}
| x^* - x_0| < \epsilon  \,.
\end{equation}

\def\epm{ e^{-\int_{t^*}^t \P(s)ds }    }
\def\epms{ e^{-\int_{t^*}^s \P(r)dr }    }
\def\epp{ e^{\int_{t^*}^t \P(s)ds }    }
We define
\begin{gather*}
\Y(t) = \X(x^*,t) - \X(x_0,t) \text{ and } \Z(t) = \X(x^*,t) + \X(x_0,t) \,, \\
\delta \mathfrak{A}(t) = \mathfrak{A}(x^*,t) -  \mathfrak{A}(x_0,t) \text{ and } \delta \mathcal{A} (t) = \mathcal{A} (x^*,t) -  \mathcal{A} (x_0,t)\,.
\end{gather*}
Then,  setting $\P(t) =\Z(t)-\mathfrak{A}(x^*,t)$,  from (\ref{Xeq}),  $ \Y(t)$ satisfies 
$$
\Y_t (t) - \P(t)  \Y(t)  =- \delta \mathfrak{A}(t) \X(x_0,t)  - \delta \mathcal{A} (t) \,,
$$
and hence
$$
\left[ \epm \Y(t)\right]_t = - \epm \left[ \delta \mathfrak{A}(t) \X(x_0,t)  + \delta \mathcal{A}(t)\right]\,.
$$
Integrating from $t^*$ to $t$, we see that
\begin{equation}\label{Yformula}
 \Y(t) =  \epp \left( \Y(t^*)-\int_{t^*}^t \epms  \left[ \delta \mathfrak{A}(s) \X(x_0,s)  + \delta \mathcal{A}(s)\right]\, ds \right)\,.
\end{equation} 

Our goal is to show that $\Y(t) \ge 0$, for all $t\ge t^*$.  
By (\ref{chi_est0}) and (\ref{assume10}), we see that 
\begin{equation}
\label{kappa2}
\Y(t^*) > 1\,,
\end{equation}
so all we need to prove is that the second term on the right-hand side of (\ref{Yformula}),
\begin{equation}
\label{kappa}
\kappa(t^*,t)=-\int_{t^*}^t \epms  \left[ \delta \mathfrak{A}(s) \X(x_0,s)  + \delta \mathcal{A}(s)\right]\, ds\,,
\end{equation} 
is very small for $t^*$ and $t$ close to $T$.

We first consider $-\int_{t^*}^s \P(r)dr$ which is equal to $-\int_{t^*}^s \Z(r)dr + \int_{t^*}^s \mathfrak{A}(x^*,r)dr$.  Since $\X(x^*,t)$ is positive,
we see that $\Z(t) > \X(x_0,t)$ and so $-\Z(t) < -\X(x_0,t)$, and as we noted above, $\X(x_0,t) > (1- \epsilon ) \delta t ^{-1} $.  Hence
$-\int_{t^*}^s \Z(r)dr < -\int_{t^*}^s \X(x_0,r)dr$, so that
$$
e^{- \int_{t^*}^s \Z(r)dr}<
e^{- \int_{t^*}^s \X(x_0,r)dr} \le e^{- \int_{t^*}^s\frac{1- \epsilon }{T -r}dr}  =\left[\frac{ T-s}{T-t^*}\right]^ { 1 - \epsilon } 
$$
and since
$e^{ \int_{t^*}^s \mathfrak{A}(x^*,r)dr}  \lesssim \mathcal{M} $,  then
$$
e^{- \int_{t^*}^s \P(r)dr}\lesssim  \mathcal{M} \left[\frac{ T-s}{T-t^*}\right]^ { 1 - \epsilon } \,.
$$

From (\ref{kappa}), we see that
\begin{align*} 
|\kappa(t^*,t)|&\lesssim \mathcal{M}   \int_{t^*}^t  
 \left[\frac{ (T-s)}{T-t^*}\right]^ { 1 - \epsilon }\left( \frac{1+ \epsilon }{T-s} \delta \mathfrak{A}(s) + \delta \mathcal{A}(s)\right) ds \\
& 
\lesssim \frac{ \mathcal{M} (1+ \epsilon )}{(T-t^*)^{1- \epsilon }}  \int_{t^*}^t  (T-s)^{ - \epsilon } \delta \mathfrak{A}(s)ds 
+ \frac{\mathcal{M} }{(T-t^*)^{1- \epsilon }}  \int_{t^*}^t   (T-s)^ { 1 - \epsilon }\delta \mathcal{A}(s)ds \,.
 \end{align*} 
Let $\vec {r}$ denote a unit-speed parameterization of the path $\gamma\subset \Gamma$ starting at $x_0$ and ending at $x^*$.
From (\ref{Acoeff}), $ \mathfrak{A}(x,t) = 2 G {v^+}' \cdot \tau (x,t) $, so that  thanks to our assumed bounds (\ref{bounds}), we see that
$$
\delta \mathfrak{A}(t) = \int_\gamma \nabla \mathfrak{A} \cdot d\vec{r} \lesssim  \mathcal{M}  | x^* - x_0| \lesssim \mathcal{M}  \epsilon \,,
$$
the last inequality following from (\ref{xx0}).   It follows that
\begin{align*} 
|\kappa(t^*,t)| &\lesssim
\frac{ \epsilon \mathcal{M} (T-t)^{1- \epsilon }}{(T-t^*)^{1- \epsilon }}  +  \epsilon \mathcal{M} 
+\frac{  \mathcal{M} (T-t)^{2- \epsilon }}{(T-t^*)^{1- \epsilon }}  + \mathcal{M}  (T-t^*) \\
& \lesssim \mathcal{M}  \left[ \epsilon + (T-t^*)\right] \ \ \ \forall t \in [t^*,T) \,.
\end{align*} 
Hence, for $T-t^*$ sufficiently small, and $t\in [t^*,T)$, we have $|\kappa(t^*,t)|<1$. Thanks to (\ref{kappa2}), this implies that for such any such $t^*$,  and for all 
$t\in [t^*,T)$, $\Y(t)\ge 0$, which by the definition of $\Y(t)$,  implies that
\begin{equation*}
\X(x^*,t)\ge \X(x_0,t)\,,
\end{equation*}
and thus $\lim_{t\rightarrow T} \X(x^*,t)=\infty$. Now, from our assumption of a single splash contact in this section, this implies that either $x^*=x_0$ or $x^*=x_1$ or $x^* = x_i$.
Since $x^*$ is sequence in $\Gamma$ converging to $x_0$, we then have $x^*=x_0$. 
Thus, by (\ref{chi_est0}) and (\ref{assume10}), we then have
\begin{equation*}
1<0\,,
\end{equation*}
which is the contradiction needed to establish that our assumption (\ref{assume10}) was wrong. 
   
By definition of $\X(x,t)$, this then shows that $\sup_{y\in \Gamma(t)}
|\nabla_{\ttt} \delta u \cdot\tt ( \cdot , t)| \le \frac{1+ (1+2c_6)(T-t)}{T-t}$ for all $t\in [ t_0(\epsilon ),T)$ with $T-{ t_0(\epsilon )}$ taken sufficiently small.   Together  with our assumed bounds (\ref{bounds}) on $u^+$, 
this   completes the proof.
\end{proof}

\section{The interface geometry near the assumed blow-up}\label{sec_geometry}

For the sake of contradiction, we assume the existence of two points $x_0$ and $x_1$ in the reference interface 
$\Gamma$  at time $t=0$  which  evolve towards
a splash singularity at time $t=T$; namely $\eta(x_0,T)=\eta(x_1,T)$.  \textcolor{black}{ In this section,} we assume that this is the only point of self-intersection 
at time $t=T$, and that no self-intersection of $\Gamma(t)$ occurs for any $t<T$. 
There may indeed also exist additional  points $x_i\in \Gamma$,  such that 
$\eta(x_i,T)=\eta(x_0,T)=\eta(x_1,T)$, but  in the course of our analysis,  we will prove that there can only be
a finite number of such points.    In the case that these
additional points $x_i \in \Gamma$ exist, we moreover show that we can  relabel the point $x_1$ so that for time $t$ sufficiently close to $T$, 
the points 
$\eta(x_1,t)$  and $\eta(x_0,t)$ are such that the vertical open segment joining $\eta(x_0,t)$ to a small neighborhood of $\eta(x_1,t)$ on $\Gamma(t)$ is contained in $\Omega^-(t)$,
as we depict in Figure \ref{fig_local}.   We will then prove that our assumption
of a finite-time splash singularity leads to a contradiction, and is hence impossible.

 If a splash singularity occurs 
at time $T$, then of course 
$\lim_{t \to T} \left|  \eta(x_0,t)-\eta(x_1,t)\right| =0$.   In this section, we find the evolution equation for the distance between the two  points
$\eta(x_0,t)$ and $\eta(x_1,t)$.

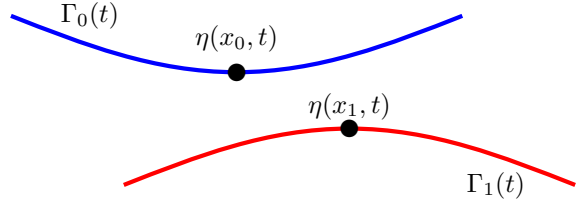
\begin{figure}[h]
 \begin{tikzpicture}[scale=1.5]
            \draw[ultra thick, blue] (4,1.) sin (6,.5);    
    \draw[ultra thick, blue] (6,.5) cos (8,1.);    
      \draw[ultra thick, red] (5,-.5) sin (7,0);    
    \draw[ultra thick, red] (7,0) cos (9,-.5);    
     

      \draw (6,.8) node { $\eta(x_0,t)$}; 
        \draw (6.,.5) node { $\newmoon$}; 
       
          \draw (7,.2) node { $\eta(x_1,t)$}; 
        \draw (7,0) node { $\newmoon$}; 
      
        \draw (4.7,1.) node { $\Gamma_0(t)$}; 
      \draw (8.3,-.5) node { $\Gamma_1(t)$};
  \end{tikzpicture} 
  \caption{{\small For $t$ sufficiently close to $T$, the interface $\Gamma(t)$ has a local neighborhood of $\eta(x_0,t)$ called 
  $\Gamma_0(t):= \eta( \gamma_0( \epsilon ), t) $  and a local neighborhood
  of $\eta(x_1,t)$ called $\Gamma_1(t):= \eta( \gamma_1( \epsilon ), t) $.  }}\label{fig_local}
\end{figure}

Recall that the
 tangent and normal directions to $\Gamma(T)$ at  $\eta(x_0,T)=\eta(x_1,T)$
 are given by the standard basis vectors $\eo=(1,0)$ and $\et=(0,1)$, respectively.
 In what follows, we will consider $0< \epsilon\ll 1$ fixed \textcolor{black}{ and sufficiently small}.
 

 With $\Gamma$ denoting the initial interface at time $t=0$ and $0 < \epsilon \ll 1$, 
recall the definition of the two small neighborhoods $\gamma_0(\epsilon)\subset \Gamma$ and $\gamma_1(\epsilon)\subset \Gamma$ given in
(\ref{eps_def}) and (\ref{eps_def2}), respectively.   According to these definitions, we may fix $ \epsilon >0 $ sufficiently small so that   for each 
$x\in \gamma_0(\epsilon)\cup\gamma_1(\epsilon)$ and for all $t\in [t_0(\epsilon),T]$,  
$\nn(\eta(x_i,t),t)$ and $\tt(\eta(x_i,t),t)$ ($i=0,1$) are almost parallel with 
$\et$ and $\eo$, respectively; in particular, the inequalities given in (\ref{eps_def}) and (\ref{eps_def2}) {\it provide a quantitative estimate} for the term
``almost parallel.''  Hence, \textcolor{black}{ by the definition of (\ref{eps_def}) and (\ref{eps_def2})},
$$\Gamma_0(t):=\eta(\gamma_0(\epsilon),t) \text{ and } \Gamma_1(t):=\eta(\gamma_1(\epsilon),t)$$
are almost flat neighborhoods  of $\eta(x_0,t)$ and $\eta(x_1,t)$ for all  $t\in [t_0(\epsilon),T]$.

Next, we define
\begin{equation}\label{aom2}
\dn(t)=\eta(x_0,t) - \eta(x_1,t)\ \text{ and } \ \du (t) = \um ( \eta(x_0,t),t) - \um ( \eta(x_1,t),t) \,,
\end{equation} 
and
$$ \dn_1 = \dn \cdot \eo\,, \  \dn_2 = \dn \cdot \et \ \text{ and } \  \duo = \du \cdot \eo\,, \  \dut = \du \cdot \et\,.
$$

Since $\eta$ is the flow of the velocity $\um$,  we see that
for any $t\in [t_0( \epsilon ),T)$,
\begin{equation}
\partial_t \dn=\um(\eta(x_0,t),t)-\um(\eta(x_1,t),t)\,.  \label{deltaeta}
\end{equation}

\begin{definition}[Distance function on $\Gamma(t)$]
We denote by $d_{\Gamma(t)}(X,Y)$ the distance along $\Gamma(t)$ between two points $X$ and $Y$ of $\Gamma(t)$.
Let $\gamma_{X,Y}(t)\subset \Gamma(t)$
denote that portion of $\Gamma(t)$ connecting the points $X$ and $Y$.
\end{definition} 
In order to establish our main result, we need the following lemmas.

\begin{lemma} \label{lemma_6.1} Let $X$ and $Y$ denote two points in $\Gamma(t)$.  Then,
$$
\bigl | \tt(X,t)-\tt(Y,t)\bigr | \le \mathcal{M} d_{\Gamma(t)}(X,Y)\,,
$$
and\textcolor{black}{, if $X_1\ge Y_1$ and $\tt_1\ge 0$ on $\gamma_{X,Y}(t)$,}
$$
 \textcolor{black}{X_1 - Y_1} \ge  \min_{ Z\in \gamma_{X,Y}(t)} \tt_1(Z,t) \  d_{\Gamma(t)}(X,Y)\,.
$$
\end{lemma} 
\begin{proof} 
  Let $\theta:[0,1] \to \Gamma$ denote a $W^{4, \infty }$-class
 parameterization of the reference interface $\Gamma$.
 There exists $ \alpha (t) ,\beta(t) \in [0,1]$ such that  
 $X =\eta(\theta(\alpha(t)),t)$  and 
 $Y=\eta(\theta(\beta(t)),t)$.

We set $\tilde \eta = \eta \circ \theta$. Then for  $ \alpha (t) \le s \le \beta(s)$, we have that
$$
\tt(\tilde\eta(\alpha(t),t),t)   - \tt(\tilde\eta(\beta(t),t),t) 
=\int_{\beta(t)}^{\alpha(t)}   \frac{d}{ds}   \tt(\tilde\eta(s,t),t)      ds\,.$$
We write $ \tt(\tilde\eta(s,t),t) $ as $ \tt(\tilde\eta) (s,t)$ and employ the chain-rule to find that
 \begin{align*}
\tt(\tilde\eta(\alpha(t),t),t)   - \tt(\tilde\eta(\beta(t),t),t) 
&=\int_{\beta(t)}^{\alpha(t)}  \partial_i \tt (\tilde \eta) (s,t)\  \tilde \eta_i'(s,t)\ ds\nonumber\\
&=\int_{\beta(t)}^{\alpha(t)} \nabla _\ttt \tt (\tilde\eta)(s,t)\ \bigl|\tilde\eta'(s,t)\bigr| ds\nonumber\\
&=\int_{\beta(t)}^{\alpha(t)} H\nn (\tilde\eta)(s,t)\  \bigl|\tilde \eta'(s,t) \bigr|ds \,,
\end{align*}
where from (\ref{tan_der}),  $ \nabla _\ttt \tt (\tilde\eta)=G ( G\tilde\eta')'$ which is equal to $H\nn (\tilde\eta)$.
Therefore,  from (\ref{bounds}),
\begin{equation}
\nonumber
\bigl | \tt(X,t)-\tt(Y,t)\bigr |\le \mathcal{M} \int_{\beta(t)}^{\alpha(t)} |[\eta\circ\theta]'(s,t)| ds\le \mathcal{M}  d_{\Gamma(t)}(X,Y)
\,.
\end{equation}

Next, we have that
\begin{align} 
X_1- Y_1  =\int_{\beta(t)}^{\alpha(t)} (\eta\circ\theta)_1'(s,t)\ ds &= \int_{\beta(t)}^{\alpha(t)}
 \tt_1((\eta\circ\theta)(s,t),t) |(\eta\circ\theta)'(s,t)|\ ds \label{062215.1}\\
 & \ge  \min_{ Z\in \gamma_{X,Y}(t)} \tt_1(Z,t) \  d_{\Gamma(t)}(X,Y)\,,\nonumber
\end{align} 
\textcolor{black}{ if $X_1\ge Y_1$ and $\tt_1\ge 0$ on $\gamma_{X,Y}(t)$.}
\end{proof} 

\begin{lemma} \label{lemma6.2} For $ 0< \epsilon \ll 1$ fixed, let  
 $\gamma_1( \epsilon )$ denote the curve defined in  (\ref{eps_def2}).
 Then, for all $t \in [t_0( \epsilon ), T]$, there exist points  $X^l(t) $ and $X^r(t)$ in the curve $ \eta(\gamma_1( \epsilon ),t)$ 
 such that 
 \begin{align*} 
  \eta_1(x_1,t) -  \frac{\epsilon}{2c_7}  \le X^l_1(t)   \le \eta_1(x_1,t) -   \frac{\epsilon}{4c_7}  
 <  \eta_1(x_1,t) +    \frac{\epsilon}{4c_7}\le  X^r_1(t) \le  \eta_1(x_1,t) +   \frac{\epsilon}{2c_7} \,,
\end{align*} 
 where the constant $c_7$ is defined in (\ref{cs-c7}), $\eta_1= \eta \cdot \eo$, $X^l_1 = X^l \cdot\eo$, and $X^r_1 = X^r \cdot\eo$.
 \end{lemma}
\begin{proof} 
According to our definition (\ref{eps_def2}) of $ \gamma _1( \epsilon )$, 
\begin{equation*}
\textcolor{black}{|}\tt_1(\eta(x,t),t) \textcolor{black}{|}>1- \epsilon  \ \ \forall \ \ 
 x \in \gamma_1(\epsilon)\,, \ t\in [t_0(\textcolor{black}{\epsilon}),T) \,.
 \end{equation*} 
 \textcolor{black}{Let us assume we are in the case
  \begin{equation}\label{cs1-June19}
\tt_1(\eta(x,t),t) >1- \epsilon  \ \ \forall \ \ 
 x \in \gamma_1(\epsilon)\,, \ t\in [t_0(\textcolor{black}{\epsilon}),T) \,,
 \end{equation} 
 the other case 
 \begin{equation*}
\tt_1(\eta(x,t),t) <-1+ \epsilon  \ \ \forall \ \ 
 x \in \gamma_1(\epsilon)\,, \ t\in [t_0(\textcolor{black}{\epsilon}),T) \,,
 \end{equation*} 
 being treated in a way similar as what follows.}
 Next, let $X$ denote a 
point $ \eta(\gamma_1( \epsilon ),t)$ such that $X_1 < \eta_1(x_1,t)$, and (by fixing $ \epsilon $ even smaller if necessary)
 satisfying
\begin{equation}\label{cs2-June19}
d_{\Gamma(t)}(X, \eta(x_1,t))  =  {\frac{ \epsilon }{4 c_7(1- \epsilon) }} \,.
\end{equation} 

By  (\ref{cs1-June19}), (\ref{cs2-June19}) and Lemma \ref{lemma_6.1},  for all $t\in [t_0(\textcolor{black}{\epsilon}),T)$,
\begin{align*} 
\textcolor{black}{\eta_1(x_1,t)-X_1 } \ge (1- \epsilon) d_{\Gamma(t)}(X, \eta(x_1,t)) 
 \ge    \textcolor{black}{{\frac{ \epsilon }{4 c_7 }} }\,.
\end{align*}  
\textcolor{black}{On the other hand, by (\ref{062215.1}), we also have that
\begin{equation*}
\eta_1(x_1,t)-X_1\le d_{\Gamma(t)} (X,\eta(x_1,t))={\frac{ \epsilon }{4 c_7(1- \epsilon) }} \le \frac{\epsilon}{2c_7}\,,
\end{equation*}
for $\epsilon>0$ small enough.}
 We then set $X^l(t) = X$.  
 
 The same argument also provides the point $X^r(t)$ which is on the right of $\eta(x_1,t)$.
\end{proof}

Our next result establishes the evolution equation for $\dn(t)$.

\def\C{\mathscr C}
\def\Z{\mathscr Z}
\def\uz{z}
\def\oz{z}

\begin{theorem}[Evolution equation for $\dn(t)$] \label{thm_geom2}
With the assumed bounds (\ref{bounds}), and for $ x_0,x_1 \in \Gamma$ such that $|\eta(x_0,t) - \eta(x_1,t)| \to 0$ as $t \to T$, if
$ \left| [\nabla_{\ttt} \delta u\cdot\tt](\eta(x_0,t),t) \right| \to \infty$ as $t\to T$, 
 then for $0 < \epsilon \ll 1$  taken sufficiently small and fixed, and 
 $t_0( \epsilon ) \in [T- \epsilon , T)$, 
 we have that  for all $t \in [t_0( \epsilon ),T)$,
\begin{equation}\label{c0}
\partial_t \dn(t) = \M(t) \dn(t) \text{ where } 
\M(t)=\frac{1}{T-t} \left[\begin{array}{cc} - \beta_1(t)& \varepsilon _1(t)\\
                                         \mathcal{E} _2(t)&  \alpha _2(t)
\end{array}
\right]\,,
\end{equation} 
where the matrix coefficients  
$$ \beta _1(t), \alpha _2(t) \in [-2 \epsilon , 1 + 2c_9(T-t)]\ \  \text{ and } \ \ \varepsilon _1(t), \mathcal{E} _2(t) \in [-2 \epsilon , 2 \epsilon ]  \,, $$
and where $c_9 = 1+2c_6$, where $c_6$ is defined (\ref{ns18.c}).
\end{theorem} 
\begin{proof}

   \begin{figure}[h]
 \begin{tikzpicture}[scale=1.5]
   \draw (1,0) -- (1,-.5);
      \draw (.4,1) node { $\Gamma_0(t)$}; 
      \draw (2.4,-.7) node { $\Gamma_1(t)$}; 
    \draw[ultra thick, blue] (0,1) sin (1,0);    
    \draw[ultra thick, blue] (1,0) cos (2,1);    
      \draw[ultra thick, red] (0.8,-.8) sin (1.8,.2);    
    \draw[ultra thick, red] (1.8,.2) cos (2.8,-.8);    
    \draw (1,0) node { $\newmoon$}; 
     \draw (1,-.5) node { $\newmoon$}; 
    \draw (1.8,.2) node { $\newmoon$}; 
     \draw (1,.3) node { $\eta(x_0,t)$}; 
      \draw (2.3,.35) node { $\eta(x_1,t)$}; 
      \draw (1.4,-.7) node { $\eta({z}(t),t)$}; 
      \draw (.75,-.25) node { $\mathfrak{r}_1(t)$}; 
      
      \draw (1.5,-.2) node { $\mathfrak{r}_2(t)$}; 
      
          \draw[ultra thick, blue] (5,1.5) sin (6,.5);    
    \draw[ultra thick, blue] (6,.5) cos (7,1.5);    
      \draw[ultra thick, red] (5.7,-1.2) sin (6.7,-.2);    
    \draw[ultra thick, red] (6.7,-.2) cos (7.7,-1.2);    
    \draw  (6,.5) -- (6,-.75);
    \draw (6.5,-.5) node { $\mathfrak{r}_2(t)$}; 
   
    \draw (5.75,.15) node { $\mathfrak{r}_1(t)$}; 
      \draw (6,.8) node { $\eta(x_0,t)$}; 
        \draw (6.,.5) node { $\newmoon$}; 
         \draw (5.45,-.7) node { $\eta({z}(t),t)$}; 
        \draw (6.,-.75) node { $\newmoon$}; 
          \draw (6.7,0.1) node { $\eta(x_1,t)$}; 
        \draw (6.7,-.2) node { $\newmoon$}; 
      
        \draw (5.4,1.4) node { $\Gamma_0(t)$}; 
      \draw (7.3,-1.1) node { $\Gamma_1(t)$};
  \end{tikzpicture} 
\caption{Left: $\eta_2(x_0,t) \le\eta_2(x_1,t) $. \hspace{1 in}  Right: $\eta_2(x_0,t) >\eta_2(x_1,t) $.}\label{fig4}
\end{figure}
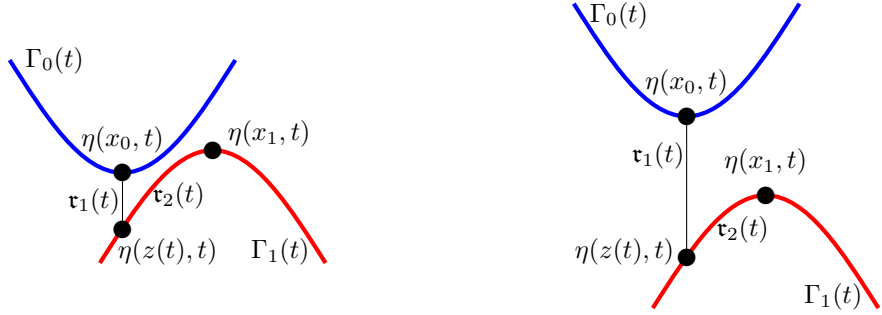

 \noindent
 {\it Step 1. The geometric set-up.}  Figure \ref{fig4}  shows the geometry of the two approaching curves at some instant of time 
 $t \in[t_0( \epsilon ),T)$: the left side of the figure shows the case that $\eta_2(x_0,t)
\le \eta_2(x_1,t)$ and the right side of the figure shows the case that $\eta_2(x_0,t)>\eta_2(x_1,t)$.\footnote{The actual curves $\Gamma_0(t)$ and $\Gamma_1(t)$ are almost flat near the assumed splash point, but we have made the slopes large to  clearly demonstrate the paths $\mathfrak{r}_1(t)$ and 
$\mathfrak{r}_2(t)$; moreover,  both $\Gamma_0(t)$ and $\Gamma_1(t)$ can have very small oscillations 
 near the contact points and do not have to be parabolas.   On the the other hand, any potential small oscillations along the curves do not  effect the qualitative picture in any way.}  Our idea is to connect $\eta(x_0,t)$ with $\eta(x_1,t)$ using a specially chosen path.  
  
We remind the reader of two facts that we shall make use of:  (1)  for $t\in[t_0(\epsilon),T)$ sufficiently small, the two approaching
 curves $\Gamma_0(t)$ and $\Gamma_1(t)$ are nearly flat, as described in (\ref{eps_def}) and (\ref{eps_def2}); (2) 
there are  two small neighborhoods $\gamma_0(\epsilon)\subset \Gamma$ and $\gamma_1(\epsilon)\subset \Gamma$ that are defined in
(\ref{eps_def}) and (\ref{eps_def2}), respectively. 
 
  We now explain why for $\epsilon>0$ chosen sufficiently small, 
  the vertical projection of $\eta(x_0,t)$ must intersect $\eta(\gamma_1({\epsilon}),t)$
at {\it one unique} point, for any $t\in[t_0(\epsilon),T]$.
Due to Lemma \ref{lemma6.2}, for $\epsilon>0$ small enough, 
 there exists a point $x \in \gamma_1({\epsilon})$ and another point $y \in  \gamma_1({\epsilon})$ such that for all $t\in [t_0( \epsilon ) ,T)$,
$$
\eta_1(x,t) + \frac{ \epsilon }{4 c_7}  \le  \eta_1 (x_1,t) \le \eta_1(y,t) -  \frac{ \epsilon }{4 c_7}\,.
$$
Now, by the fundamental theorem of calculus, 
$$
|\eta(x_1,t)-\eta(x_1,T)|\le \left| \int_t^T v^-(x_1,s)ds\right| \le \mathcal{M} (T-t)\,,
$$
where we have used Lemma \ref{lemma1} to bound $v^-$.
From (\ref{eps_def2}),  $T- t_0(\epsilon) \le \frac{ \epsilon }{ 100 c_7 \mathcal{M} }$; it follows that
$$
|\eta(x_1,t)-\eta(x_1,T)|\le \frac{ \epsilon }{ 100 c_7  } \,.
$$
Similarly, $|\eta(x_0,t)-\eta(x_0,T)|\le \frac{ \epsilon }{ 100 c_7  }$ and using that $\eta(x_0,T)=\eta(x_1,T)$, we see that (by taking
$ \epsilon $ even smaller if necessary)
$$
\textcolor{black}{\eta_1(x,t)\le\eta_1(x,t) + \frac{ \epsilon }{ 5 c_7}  \le  \eta_1 (x_0,t) \le \eta_1(y,t) -  \frac{ \epsilon }{5 c_7}\le \eta_1(y,t)}\,.
$$
\textcolor{black}{By the intermediate value theorem, this shows there exists 
$\eta(z(t),t)\in\eta(\gamma_1({\epsilon}),t)$ such that $\eta_1(z(t),t)=\eta_1(x_0,t)$}, and hence
\begin{equation}\label{cs3-June19}
\eta_1(x,t) + \frac{ \epsilon }{5 c_7}  \le  \eta_1 (z(t),t) \le \eta_1(y,t) -  \frac{ \epsilon }{5 c_7}\,.
\end{equation} 
This proves the existence of a point $\eta(z(t),t)$ in the curve $\Gamma_1(t) := \eta(\gamma_1( \epsilon ),t)$ which has the same horizontal component
as the point $\eta(x_0,t)$ for every $t \in [t_0( \epsilon) ,T)$.

Let us now show that there cannot be a second point in this intersection. 
We proceed by contradiction, and assume the existence of a different point $Z(t) \in \gamma_1( \epsilon )$ such that $Z(t)\ne z(t)$ and
satisfies (\ref{cs3-June19}). Since $\eta_1(z(t),t)= \eta_1(Z(t),t)$, by Rolle's theorem, there exists $c(t)\in\gamma_1({\epsilon})$ such that 
\begin{equation}
\label{145rev5}
 \eta_1 '(c(t),t)= 0\,.
\end{equation} 
Since for any $t<T$, $\det\nabla\eta=1$, we then  have that for $t <T$,  $|\eta'|\ne 0$, and (\ref{145rev5}) provides
\begin{equation}
\label{145rev6}
0=\frac{ \eta_1'(c(t),t)}{|\eta' (c(t),t),t)|}=\tt(\eta(c(t),t),t)\cdot \eo \,.
\end{equation}
Therefore, $\tt(\eta(c(t),t),t)=(\tt(\eta(c(t),t),t)\cdot \et) \, \et$, which with (\ref{eps_def}) provides 
$$1=|\tau(\eta(c(t),t),t) \cdot \et |\le \epsilon $$ 
which is a 
contradiction as $\epsilon<1$.

As shown in Figure \ref{fig4}, 
we define $\mathfrak{r}_1(t)$ to be the vertical line segment connecting $\eta(x_0,t) \in \Gamma_0(t)$ to $\Gamma_1(t)$.   Let us now
explain why the path $\mathfrak{r}_1(t)$  can always be assumed to  be contained in the closure of $\Omega^-(t)$.

We assume that the  path $\mathfrak{r}_1(t)$ is not contained in the closure of $\Omega^-(t)$. 
Then, since $\Omega^+(t)$ is an open and connected set, $\Omega^+(t)\cap\mathfrak{r}_1(t)$ is a union of segments 
$\mathcal{S} _i:=] \x_i(t),\y_{i}(t)[$, with $\x_i(t) > \y_i(t)$ for each $i$, and each segment $ \mathcal{S} _i$ lies strictly above
the next segment $S_{i+1}$.

We now show that there can only be a finite number of such segments $ \mathcal{S}_i$.
Let $ \mathcal{S} _i$ and $ \mathcal{S} _{i+1}$ be two such consecutive segments.    Let $ \mathfrak{c}_i (t) \subset \Gamma(t)$
denote the portion of $\Gamma(t) $ connecting the point $\y_i \in \mathcal{S} _i$ to $\x_{i+1} \in \mathcal{S} _{i+1}$.
   We denote the open set $L_i(t)\subset \Omega ^-(t)$ as the set enclosed by the curve 
$\mathfrak{c} _i(t)$ and the vertical segment $ ]\y_i(t), \x_{i+1}(t)[$, as shown in Figure \ref{fig5cs}.   The set $L_i(t)$ is either to the 
left or to the right of the vertical path $\mathfrak{r} _1(t)$.

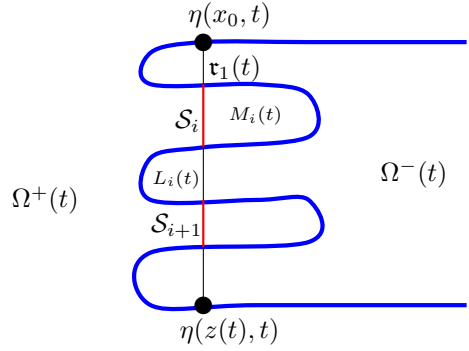
\begin{figure}[h]
\begin{tikzpicture}[scale=0.7]
    \draw (1,2) node { $\Omega^+(t)$}; 
    \draw (3.5,2.4) node {  $_{ L _i(t) }$}; 
    \draw (5,3.6) node {  $_{ M _i(t) }$}; 
    \draw (8,2.5) node { $\Omega^-(t)$}; 
    \draw[color=blue,ultra thick] plot[smooth,tension=.6] coordinates{(9,5) (6,5.) (4,5) (3,4.8) (3,4.2) (5, 4.2) (6,4)  (6,3.2)  (4,3)  (3,2.8) (3,2) (5,2) (6,2) (6,1.2) (3,1) (3,0) (5, 0)  (9,0) 
     };
     \draw (4.5,5.5) node { $\eta(x_0,t)$}; 
     \draw (4,5) node { $\newmoon$}; 
     \draw (4,0) node { $\newmoon$}; 
      \draw (4.5,-0.5) node { $\eta(z(t),t)$};

        \draw (4,5) -- (4,0);
        \draw[color=red,thick] (4,4.2) -- (4,3);
        \draw[color=red,thick] (4,2) -- (4,1.1);
         \draw (4.6,4.5) node { $\mathfrak{r}_1(t)$}; 
       \draw (3.7,3.5) node { $\mathcal{S} _i$}; 
        \draw (3.5,1.5) node { $\mathcal{S} _{i+1}$}; 
\end{tikzpicture} 
\caption{If we suppose that the vertical line segment $\mathfrak{r}(t)$, connecting $\eta(x_0,t)$ to $\eta( z(t),t)$ is not contained in
$\overline{\Omega^-(t)}$, then $\Omega^+(t)\cap\mathfrak{r}_1(t)$ consists of the union of finitely many open intervals $ \mathcal{S} _i$ (shown in red).}  \label{fig5cs}
\end{figure}

Below, we shall prove that the slope of the tangent vector to $\Gamma(t)$ at the points $\x_i(t)$ and $\y_i(t)$ cannot be too large;
specifically,  we will show that
\begin{equation}
\label{061215.1}
|\tt(\x_i,t)\cdot \eo|\ge \frac{1}{\sqrt{2}} \ \text{ and } \ 
|\tt(\y_i,t)\cdot \eo|\ge \frac{1}{\sqrt{2}}
\,.
\end{equation}

We now assume that  (\ref{061215.1}) holds, and 
as shown in Figure \ref{fig5cs},
 we assume that $L_i(t)$ is to the left of $ \mathfrak{r}_1 (t)$.  Let $\theta:[0,1] \to \Gamma$ denote a $W^{4, \infty }$-class
 parameterization of the reference interface $\Gamma$.
Let $\pp_i(t) \in \mathfrak{c} _i(t)$ denote the left-most extreme point on $\partial L_i(t)$; then,
 there exists $ \alpha (t) \in [0,1]$ such that  
 $\pp_i(t) =\eta(\theta(\alpha(t)),t)$,  and  $\nn(\pp_i(t),t)=-\eo$. 
 Let $\beta(t) \in [0,1]$ be such that $\eta(\theta(\beta(t)),t)=\y_i$.

Using Lemma \ref{lemma_6.1} and the lower-bound (\ref{061215.1}),
\begin{equation}
\label{052115.1}
\frac{1}{\sqrt{2}}\le\bigl | \tt(\eta(\theta(\alpha(t)),t),t)-\tt(\eta(\theta(\beta(t)),t),t)\bigr |\le \mathcal{M}\times \text{length of}\ \mathfrak{c} _i(t)\,.
\end{equation}
Since each loop $\mathfrak{c} _i(t)$ is of length greater than $\frac{1}{\sqrt{2}\mathcal{M}}$ and $ \mathfrak{c} _i$ is disjoint from
$ \mathfrak{c} _j$ for $i\neq j$, the fact that
 $\Gamma(t)$ is of finite length,  by (\ref{conservationlaw}), implies that the number of such loops $ \mathfrak{c} _i(t)$ is bounded; hence,
  the intersection of  $\mathfrak{r}_1(t)$  with $\Omega^+(t)$ consists  of a finite number of segments $ \mathcal{S}_i$.

Having established that this generic loop $ \mathfrak{c} _i(t)$ (shown to the left of the vertical path $ \mathfrak{r} _1(t)$ in
Figure \ref{fig5cs}) is of length greater than $\frac{1}{\sqrt{2}\mathcal{M}}$, we now turn our attention to the study of the
subset 
$M_i(t) \subset \Omega^+(t)$ which is directly to the right of $ \mathcal{S} _i$; that is, $M_i(t)$ is the open set whose boundary
consists of that portion of $\Gamma(t)$ connecting $\x_i(t)$ with $\y_i(t)$, which we call $ \mathfrak{b}  _i(t)$,  and $ \mathcal{S} _i$.


Next, let $\q_i(t) \in \mathfrak{b}  _i(t)$ denote the  right-most extreme point of $M_i(t)$;  then, similarly as for the case of $\mathfrak{c}_i(t)$,  we find that the length of $\mathfrak{b} _i(t)$ is greater than 
$\frac{1}{\sqrt{2}\mathcal{M}}$.

We now explain why the projection of the  set $M_i(t)\cup L_i(t)$ onto the horizontal axis spanned by $\eo$ has a vastly larger length
than $T-t$. 
In the same way as we obtained the inequality (\ref{052115.1}),
we have that for any $x=\eta(\theta(\kappa(t)),t)\in \Gamma(t)$ that
\begin{equation*}
\bigl | \tt(\eta(\theta(\beta(t)),t),t)-\tt(\eta(\theta(\kappa(t)),t),t)\bigr |\le \mathcal{M}\times d_{\Gamma(t)}(x,\y_i(t))\,.
\end{equation*}
Therefore, with  $| \Gamma (t)|$ denoting the length of $\Gamma(t)$,
 for any 
$x\in\Gamma(t)$ such that $d_{\Gamma(t)}(x,\y_i(t))\le \min( {\frac{1}{2}} |\Gamma(t)| , \frac{1}{2\sqrt{2}\mathcal{M}})$, we have that 
\begin{equation*}
\bigl |\tt(\eta(\theta(\kappa(t)),t),t)\cdot \eo \bigr |\ge \frac{1}{2\sqrt{2}}\,.
\end{equation*}
We can assume that
\begin{equation}
\label{061215.2}
\tt(\eta(\theta(\kappa(t)),t),t)\cdot \eo \ge \frac{1}{2\sqrt{2}}\,,
\end{equation} 
for the case with the opposite sign  can be treated in a similar fashion (as what follows).

Then, for any $\kappa(t)>\beta(t)$ such that $x=\eta(\theta(\kappa(t)),t)$ satisfies 
$$d_{\Gamma(t)}(x,\y_i(t))=\min \left( {\frac{1}{2}} | \Gamma(t)|, \frac{1}{2\sqrt{2}\mathcal{M}}\right)\,,$$
 we have, by Lemma \ref{lemma_6.1}
and  the inequality (\ref{061215.2}), that
\begin{align} 
(x-\y_i(t))\cdot  \eo \ge \frac{1}{2\sqrt{2}} d_{\Gamma(t)}(x,\y_i) 
=\frac{1}{2\sqrt{2}} \min\left( {\frac{1}{2}} |\Gamma(t)|, \frac{1}{2\sqrt{2}\mathcal{M}}\right)>0\,,
\label{061215.3}
\end{align} 
which shows that $\mathfrak{b}_i(t)$ extends to the right of $\mathfrak{r}_1(t)$ by a distance of  at least 
$$\frac{1}{2\sqrt{2}} \min\left({\frac{1}{2}} | \Gamma(t)|, \frac{1}{2\sqrt{2}\mathcal{M}}\right)>0$$
in the $\eo$-direction.
Using the identical argument, we can prove that $\mathfrak{c}_i(t)$ extends to the left of $\mathfrak{r}_1(t)$ by  a distance of 
at least $\frac{1}{2\sqrt{2}} \min({\frac{1}{2}} |\Gamma(t)|, \frac{1}{2\sqrt{2}\mathcal{M}})>0$ in the $-\eo$-direction.

We now prove the inequalities in (\ref{061215.1}).   We shall consider the tangent vector $\tt$ at $ \y_i(t)$, as the proof for $\tt$ at
 $\x_i(t)$ is
identical.  For the sake of contradiction, we assume that
\begin{equation}\label{cs_contradiction1}
|\tt(\y_i(t),t)\cdot \eo| <  \frac{1}{\sqrt{2}}\,,
\end{equation} 
so that
$$
|\tt(\y_i(t),t)\cdot \et |\ge    \frac{1}{\sqrt{2}}\,.
$$
We choose a point $x \in \mathfrak{b} _i(t)$ which is either to the left or to the right of $\y_i(t)$ such that 
\begin{equation}\label{061715.0}
d_{\Gamma(t)}(x,\y_i(t)) \textcolor{black}{=} \min \left( \frac{1}{3}|\Gamma(t)|,  \frac{1}{2\sqrt{2}\mathcal{M}}\right)\,.
\end{equation} 
In the same way as we obtained  (\ref{061215.3}), we see that if we choose $x$ to be on the correct side of $\y_i(t)$ (depending on the sign of 
$\tt(\y_i(t),t)\cdot \et$), we have that
\begin{equation}
\label{062315.0}
(x-\y_i(t))\cdot  \et \ge \frac{1}{2\sqrt{2}}\ d_{\Gamma(t)}(x,\y_i(t))\,,
\end{equation}
as well as
$$
|(x-\y_i(t))\cdot  \eo| \le
\frac{3}{2\sqrt{2}}\ d_{\Gamma(t)}(x,\y_i(t))\,,
$$ 
so that $x$ is in the cone with vertex $\y_i(t)$ given by
\begin{equation}
\label{061715.1}
(x-\y_i(t))\cdot  \et \ge \frac{1}{3}\ |(x-\y_i(t))\cdot \eo|\,.
\end{equation} 
Furthermore, using (\ref{062315.0}), we have that
\begin{equation}
\label{061715.2}
(x-\y_i(t))\cdot  \et \ge\frac{1}{2\sqrt{2}  }\min \left( \frac{|\Gamma(t)|}{3},  \frac{1}{2\sqrt{2}\mathcal{M}}\right)\,.
\end{equation} 
Therefore, we have just established the existence of  $\tilde{\mathfrak{c}} _i(t)\subset \Gamma(t)$, such that
$ \tilde{ \mathfrak{c}} _i(t)$ is the shortest curve which connects $\y_i(t)$ to $x$ and satisfies
$$
 \text{length of } \tilde{ \mathfrak{c}} _i(t) =  \min \left(\frac{|\Gamma(t)|}{3},  \frac{1}{2\sqrt{2}\mathcal{M}}\right)
$$
which is bounded from below by a positive constant  as $t \to T$.
Moreover, the curve $ \tilde{ \mathfrak{c} }_i(t)$ is contained in the cone defined in
 (\ref{061715.1}), whose vertex $\y_i(t)$ satisfies
 $$
 |\eta(x_0,t) - \textcolor{black}{\y_i(t)}| <  |\eta(x_0,t) - \eta(z(t),t)|\,,
 $$
 the right-hand side tending to zero as $t\rightarrow T$, \textcolor{black}{since as $t \to T$,  $\eta(z(t),t)\rightarrow \eta(x_1,T)$ which implies that $\lim_{t\rightarrow T} |\eta(x_0,t)-\eta(z(t),t)|=0$.
 To sum up, $\tilde{\mathfrak{c}}_i(t)$ is a curve of length of order $1$, of positive vertical extension above $\y_i(t)$ of order $1$, and is contained in the cone (\ref{061715.1}) with vertex $\y_i(t)$ which is below $\eta(x_0,t)$ (and the distance between these two points converges to zero as $t\rightarrow T$).
 }

On the other hand, since $\tt(\eta(x_0,T),T)\cdot \et=0$, we have in the same manner that for $T-t$ sufficiently small,
there exists a curve $\tilde\Gamma_0(t) \subset \Gamma(t)$ containing the point $\eta(x_0,t)$, and 
of length $\min \left( \frac{1}{2} |\Gamma(t)|,  \frac{1}{200\mathcal{M}}\right)$ \textcolor{black}{such that}  the curve
 $\tilde\Gamma_0(t)$ is contained in the two cones (\textcolor{black}{that are almost horizontal from the definition below}) defined by
\begin{equation}
\label{061715.3}
|(x-\eta(x_0,t))\cdot\et|\le\frac{1}{100} |(x-\eta(x_0,t))\cdot\eo|\,;
\end{equation}
additionally, 
the curve  $\tilde\Gamma_0(t)$ extends in the $\pm\eo$ direction a distance of at least 
$\frac{1}{2} \min \left( \frac{1}{2} |\Gamma(t)|,  \frac{1}{200\mathcal{M}}\right)$ on each side of $\eta(x_0,t)$.
It is then elementary to show that the cone given by (\ref{061715.1}) intersects each of the four lines enclosing the cone 
(\ref{061715.3}) at a distance which less than $\frac{1}{2} \min \left( \frac{1}{2} |\Gamma(t)|,  \frac{1}{200\mathcal{M}}\right)$ on 
each side of $\eta(x_0,t)$. Therefore, the cone given by (\ref{061715.1}) intersects $\tilde\Gamma_0(t)$. 
The same is true for the curve $\tilde{ \mathfrak{c} }_i(t)$, as its starting point $\y_i(t)$ lies below  $\tilde\Gamma_0(t)$ while, 
due to (\ref{061715.2}),
its ending point $x$ lies above $\tilde\Gamma_0(t)$ for $t$ sufficiently close to $T$, and stays in the cone given by (\ref{061715.1}). Furthermore,
this self-intersection occurs with different tangent vectors, since thanks to Lemma \ref{lemma_6.1},  any point $z$ on $\tilde \Gamma_0(t)$ satisfies (for $t$ close enough to $T$)
$$|\tt(z,t)\cdot \et|\le \frac{1}{100}\,,$$ while any point $z$ on $\tilde{\mathfrak{c}}_i(t)$ will satisfy thanks to Lemma \ref{lemma_6.1} that
$$|\tt(z,t)\cdot \et|\ge \frac{1}{2\sqrt{2}}\,.$$

As $\Gamma(t)$ cannot self-intersect for $t<T$ (particularily not with different tangent vectors), this then leads to a contradiction  
of (\ref{cs_contradiction1}), and hence proves (\ref{061215.1}).

Let $\bar \gamma \subset \Gamma$ be the preimage of $\eta$ of  the loops
 $\mathfrak{b}  _i( \cdot , t_0( \epsilon ) )\cap  \mathfrak{c}  _i( \cdot , t_0( \epsilon ) )$.
It follows that for  all $t \in (t_0( \epsilon ) , T)$,  $\eta( \bar \gamma, t)$ must continue to intersect the vertical
path $ \mathfrak{r} _1(t)$ at the finite set of points $\x_i(t)$ \textcolor{black}{ and $\y_i(t)$}. 
 Since $\eta_2(x_0,t)>\x_i(t)>\eta_2(x_1,t)$ (the same being true for $\y_i(t)$), we still have by continuity that $\eta_2(x_0,t')>\x_i(t')>\eta_2(x_1,t')$ (the same being true for $\y_i(t')$) for $t'\in [t,T)$,  as the case $\eta_2(x_0,t')=\x_i(t')$ or $\eta_2(x_1,t')=\x_i(t')$ 
 correspond to a self-intersection of $\Gamma(t')$ at time $t'<T$, which is excluded from our definition of $T$. 
 
 This ensures that the already established finite number $ \mathcal{N} $ of loops $\mathfrak{b}_i(t)$ and $\mathfrak{c} _i(t)$ stays constant for 
 $t\in [t_0(\epsilon),T]$ for $T-t_0(\epsilon)>0$ small enough.  We then have the existence of a finite number of points $x_0, x_1, x_2,..., x_n$ in $\Gamma$ such that 
 $$\eta(x_0,T)=\eta(x_1,T)= \cdot \cdot\cdot =\eta(x_n,T)$$
 and such that $\eta(x_i,t)$ for $i\in [2,\mathcal{N} ]$ belongs to the image by the flow of the same loop (of length of at least $\frac{1}{\sqrt{2}} \min(\frac{1}{2} |\Gamma(t)|, \frac{1}{2\sqrt{2}\mathcal{M}})$ on each side of a corresponding point of intersection of $\mathfrak{r}_1(t)$ and $\Gamma(t)$) for all $t\in [t_0(\epsilon),T]$).
 We can then, if necessary, replace  the point $x_1$ by an appropriate $x_i \in \Gamma$ (with $\eta(x_i,t)$ such that $\eta(x_0,t)$ and $\eta(x_i,t)$ are on the same loop for all $t\in [t-t_0(\epsilon), T]$). Therefore, the vertical path $\mathfrak{r}_1(t)$, connecting $\eta(x_0,t)$ to $\Gamma_1(t)$, is contained in 
 $\Omega^-(t)$ for $T-t$ sufficiently small. In what follows, we assume that this substitution has been made so that $x_1 \in \Gamma$
 is the point which is assumed to flow into self-intersection from below (by renaming $x_0$ and $x_1$ if necessary).



 We can therefore define the unique point $z(t) \in \Gamma$ such that $\eta(z(t),t)$ is the vertical projection of $\eta(x_0,t)$ onto the curve $\mathfrak{r}_1(t)$ (as shown in
 Figure \ref{fig4}).
Specifically,  we define $\mathfrak{r}_1(t)$ to be the vertical line segment connecting $\eta(x_0,t) \in \Gamma_0(t)$ to $\eta(z(t),t) \in \Gamma_1(t)$ (which is contained in $\overline{\Omega^-}(t)$ as we just have shown), and we
define $\mathfrak{r}_2(t)$ to be the portion of \textcolor{black}{$\Gamma_1(t)$} connecting  $\eta(z(t),t)$ to $\eta(x_1,t)$.

\vspace{.2 in}

We will rely on the following two claims:   

\vspace{.05 in}

\noindent
{\it Claim 1.}  For $t \in [t_0( \epsilon ), T)$, $\eta_2(x_1,t) - \eta_2(z(t),t) =  b(t) \dn_1(t)\,\left[ (T-t) + | \dn_1(t)| \right]$ for a bounded function $b(t)$.
 \begin{proof} 
 Near the point $\eta(x_1,t)$, we consider $\mathfrak{r} _2(t)$ as a graph $(X, h(X,t))$ (see Figure \ref{fig4}), 
 such that $h(0,t) = \eta_2(x_1,t)$ with  tangent
 vector $(1, h'(X,t))$, which  at $X=0$ must be close to horizontal, since $h'(0,T)=0$.
 Since $h$ is a $C^2$ function, we can write the Taylor series for $h(X,t)$ about $X=0$ as
 \begin{equation}
 \label{052615.0}
 h(X,t) = h(0,t) + h'( 0,t ) X +\frac{1}{2} h''(\xi,t) X^2 \text{ for some }  \xi \in (0, X) \,.
 \end{equation}
 Next,
 \begin{align}
 |h'( 0 ,t) |& =\left| h'( 0 ,T) + \int_T^t  h'_t( 0 ,s)ds \right|  \nonumber\\
 &= \left| \int_T^t  {v^-_2}'( x_1, s) ds \right| 
 \lesssim \mathcal{M}  (T-t) \,,\label{052615.1}
\end{align} 
the inequality following from the bound on $v^-$ given by Lemma \ref{lemma1}. On the other hand,
\textcolor{black}{ 
\begin{align}
|h''(\xi,t)|& = |H(\xi,h(\xi,t))\ (1+h'^2(\xi,t))^ {\frac{3}{2}} |\nonumber\\
&=\left|H(\xi,h(\xi,t))\ \left(1+\left[\frac{\tt(\xi,h(\xi,t))\cdot \et}{\tt(\xi,h(\xi,t))\cdot \eo}\right]^2\right)^ {\frac{3}{2}}  \right|\nonumber\\
&\le |H(\xi,h(\xi,t)) | \left( 1 + \frac{ \epsilon ^2}{ (1- \epsilon )^2}  \right)^ {\frac{3}{2}} 
\lesssim \mathcal{M}\,,\label{052615.2}
\end{align}
}
where we have used (\ref{eps_def2}) for the first inequality and (\ref{bounds}) for the second.   From (\ref{052615.0}), (\ref{052615.1}) and (\ref{052615.2}), we then have that
\begin{equation}
 \label{052615.3}
 |h(X,t) -h(0,t)|\le C \mathcal{M} |X| (T-t+|X|)\,,
 \end{equation}
 for some constant $C>0$.
 
Next, we notice that $\eta_2(z(t),t) = h( \dn_1(t),t)$; hence, we set $X =  \dn_1(t)$.   By setting $b(t) = C \mathcal{M} \vartheta(t)$ with  $ \vartheta(t) \in (0,1)$, the proof is complete.
 \end{proof} 
 
 \noindent
{\it Claim 2.} $|\dn_1(t) |  \lesssim  \mathcal{M} (T-t) < \epsilon $ for $t \in [t_0( \epsilon ),T)$.
 \begin{proof} 
 By the fundamental theorem of calculus, $|\dn_1(t) | \le \int_T^t | \delta v(s)|  ds \lesssim  \mathcal{M} (T-t)$ by Lemma \ref{lemma1}.   Then, we choose
 $T-t_0( \epsilon )$ sufficiently small.
 \end{proof}

 \vspace{.1 in}
 \noindent
{\it Step 2. The case that $ \eta_2(x_0,t) > \eta_2(x_1,t)$.}  We will first consider the geometry displayed on the right side of Figure \ref{fig4}.
With $\vr_1(t)$ and $\vr_2(t)$ denoting unit-speed parameterizations for $\mathfrak{r}_1(t)$ and $\mathfrak{r}_2(t)$, 
\begin{align}
\umo(\eta(x_0,t),t)-\umo(\eta(x_1,t),t)
&=\left[\umo(\eta(x_0,t),t)-\umo(\eta(\oz (t),t),t)\right]  \n  \\
&\qquad \qquad
+\left[\umo(\eta(\oz (t),t),t)-\umo(\eta(x_1,t),t)\right]\n\\
&=\int_{\mathfrak{r}_1(t)}  \nabla \umo \cdot  d\vr_1+\int_{\mathfrak{r}_2(t)}  \nabla \umo \cdot  d\vr_2 \n \\
&=\int_{\mathfrak{r}_1(t)}  \frac{\partial u^-_2}{\partial x_1} \, dx_2 +  \int_{\mathfrak{r}_2(t)} \nabla _\ttt u^-_1 \, ds\,, \n
\end{align} 
where we have used the fact that $\frac{\partial u^-_1}{\partial x_2} = \frac{\partial u^-_2}{\partial x_1}$ in the last equality, as $  \operatorname{curl} \um=0$.
  We will evaluate these two
integrals using the mean value 
theorem for integrals, together with our estimate  (\ref{improved}) for $ \nabla _\ttt \um \cdot \tt$, and hence for $ \frac{\partial u^-_1}{ \p x_1}$ (which is equivalent to $ \nabla _\ttt u^-$ for $T-t_0( \epsilon )$ 
sufficiently small, as the ratio of the two quantities is close to $1$), and estimate (\ref{u12epsilon}) for  $\frac{\p u^-_2}{\p x_1}$.  In particular, 
\begin{align} 
&\umo(\eta(x_0,t),t)-\umo(\eta(x_1,t),t)\n \\
&\qquad
=\frac{\varepsilon_1(t)}{T-t} \left( \eta_2(x_0,t) - \eta_2(z(t),t) \right)- \varrho(t)\frac{\alpha _1(t)}{T-t} \dn_1(t) -  \nu(t) \frac{\alpha_1(t)}{T-t}  \left( \eta_2(x_1,t) - \eta_2(z(t),t) \right)\,,\n\\
&\qquad
=\frac{\varepsilon_1(t)}{T-t} \dn_2(t) + \frac{\varepsilon_1(t)}{T-t} \left(\eta_2(x_1,t) - \eta_2(z(t),t) \right) - \varrho(t)\frac{\alpha _1(t)}{T-t} \dn_1(t)  \n \\
&\qquad\qquad \qquad\qquad  \qquad\qquad  -  \nu (t)\frac{\alpha_1(t)}{T-t}  \left( \eta_2(x_1,t) - \eta_2(z(t),t) \right),
\label{j1cs}
\end{align}
where $\varepsilon_1(t)\in [-\epsilon,\epsilon]$, and  where we choose $\alpha_1(t) \in [-\epsilon,1+ c_9(T-t)]$, where  $0< \epsilon \ll 1$ is defined in 
Step 4 of the proof of Theorem \ref{thm_blowup}.   The functions $\varrho(t)$ and $\nu(t)$ satisfy $| 1-\varrho(t)|\ll 1$ and $0\le \nu(t)\ll 1$; this follows
since $\mathfrak{r}_2(t)$ is nearly flat near $\eta(x_0,t)$, so the vertical distance $|\eta_2(x_1,t) - \eta_2(z(t),t)|$ is nearly zero, while the horizontal
distance $|\eta_1(x_1,t) - \eta_1(z(t),t)|$ is nearly  the total distance $|\eta(x_1,t) - \eta(z(t),t)|$.

The negative sign in front of $ \alpha _1(t)$ is determined by the limiting behavior of $\frac{\partial u^-_1}{\partial x_1}$ given by (\ref{ns19}).  
From Claim 1 above, we then see that
\begin{align} 
&\umo(\eta(x_0,t),t)-\umo(\eta(x_1,t),t)\n \\
&\qquad
=\frac{\varepsilon_1(t)}{T-t} \dn_2(t) + \frac{b(t)(|\dn_1(t)| + \delta t) \varepsilon_1(t)}{T-t} \dn_1(t) -\frac{ \varrho\alpha _1(t)}{T-t} \dn_1(t) \n \\
& \qquad\qquad  \qquad\qquad 
-  
 \frac{\nu b(t) (|\dn_1(t)| + \delta t)\alpha_1(t)}{T-t} \dn_1(t) \,,\n
\end{align}
where $  \delta t = T-t$.
We set 
$$\beta_1(t) =  \left[\varrho(t) + \nu(t) b(t)(|\dn_1(t)| + \delta t)\right] \alpha _1(t)-b(t)(|\dn_1(t)| + \delta t)\varepsilon_1(t).$$
 Then, with Claim 2, we see that
$\beta_1(t) \in [-2 \epsilon , 1+ 2c_9(T-t) ]$, and that 
\begin{align} 
\umo(\eta(x_0,t),t)-\umo(\eta(x_1,t),t)
=- \frac{\beta_1(t)}{T-t} \dn_1(t)+\frac{\varepsilon_1(t)}{T-t} \dn_2(t)  \label{j1} \,.
\end{align}
  
Similarly, for $u_2^-$,  we have that
 \begin{align}
\umt(\eta&(x_0,t),t)-\umt(\eta(x_1,t),t)
=\left[\umt(\eta(x_0,t),t)-\umt(\eta(\oz (t),t),t)\right]  \n  \\
&\qquad \qquad\qquad \qquad\qquad \qquad\qquad \qquad
+\left[\umt(\eta(\oz (t),t),t)-\umt(\eta(x_1,t),t)\right]\n\\
&=\int_{\mathfrak{r}_1(t)}  \nabla \umt \cdot  d\vr_1+\int_{\mathfrak{r}_2(t)}  \nabla \umt \cdot  d\vr_2 \n \\
&=\int_{\mathfrak{r}_1(t)}  \frac{\partial u^-_2}{\partial x_2} \, dx_2 +  \int_{\mathfrak{r}_2(t)} \nabla _\ttt u^-_2 \, ds\,, \n\\
&=\frac{\alpha_2(t)}{T-t} \left( \eta_2(x_0,t) - \eta_2(z(t),t) \right)+ \varrho(t)\frac{\varepsilon _2(t)}{T-t} \dn_1(t) +  \nu(t) \frac{ \alpha  _2(t)}{T-t}  \left( \eta_2(x_1,t) - \eta_2(z(t),t) \right)\,, \n\\
&=\frac{\alpha_2(t)}{T-t}\dn_2(t)+ \frac{b(t)(|\dn_1(t)| + \delta t)\alpha_2(t)}{T-t} \dn_1(t) \n \\
& \qquad \qquad + \frac{\varrho(t)\varepsilon _2(t)}{T-t} \dn_1(t) +   \frac{\nu(t)b(t) (|\dn_1(t)| + \delta t)\alpha  _2(t)}{T-t} \dn_1(t) \,, 
\n
\end{align}
 with $\varepsilon_2(t)\in [-\epsilon,\epsilon]$ and $\alpha_2(t)\in [-\epsilon,1+c_9(T-t)]$, and where $0\le 1-\varrho(t)\ll 1$ and $0\le \nu(t)\ll 1$.   Setting
\begin{equation}\label{zzt400}
 \mathcal{E}_2(t) =  ( b(t) + \nu(t)b(t))(|\dn_1(t)| + \delta t)\alpha  _2(t)+   \varrho(t) \varepsilon_2(t)
\end{equation}
 we see that by Claim 2, 
  \begin{equation}
\label{718.-1}
\mathcal{E}_2(t) \in [-2 \epsilon , 2 \epsilon ]  \,,
\end{equation}
and
 \begin{equation} 
 \umt(\eta(x_0,t),t)-\umt(\eta(x_1,t),t)
 =\frac{ \mathcal{E} _2(t) }{T-t}\dn_1(t) +\frac{ \alpha _2(t) }{T-t}\dn_2(t) \,. \label{j2}
 \end{equation} 
 Equations (\ref{deltaeta}),  (\ref{j1}) and (\ref{j2}), then give the desired relation (\ref{c0}).

  \vspace{.15 in}
 \noindent
{\it Step 3. The case that $ \eta_2(x_0,t) \le \eta_2(x_1,t)$.} We next consider the geometry displayed on the left side of Figure \ref{fig4}.
Again, using $\vr_1(t)$ and $\vr_2(t)$ to denote unit-speed parameterisations for $\mathfrak{r}_1(t)$ and $\mathfrak{r}_2(t)$, we see that once again
\begin{align}
\umo(\eta(x_0,t),t)-\umo(\eta(x_1,t),t)
&=\left[\umo(\eta(x_0,t),t)-\umo(\eta(\uz (t),t),t)\right]  \n  \\
&\qquad \qquad
+\left[\umo(\eta(\uz (t),t),t)-\umo(\eta(x_1,t),t)\right]\n\\
&=\int_{\mathfrak{r}_1(t)}  \frac{\partial u^-_2}{\partial x_1} \, dx_2 +  \int_{\mathfrak{r}_2(t)} \nabla _\ttt u^-_1 \, ds\,, \n
\end{align} 
where $s$ denotes arc length. 
   We again evaluate these two
integrals using the mean value 
theorem for integrals:
\begin{align} 
&\umo(\eta(x_0,t),t)-\umo(\eta(x_1,t),t)\n \\
&\ \
=
\frac{\varepsilon_1(t)}{T-t} \left( \eta_2(x_0,t) - \eta_2(z(t),t) \right)-\frac{ \varrho(t)\alpha _1(t)}{T-t} \dn_1(t) -   \frac{\nu(t)\alpha_1(t)}{T-t}  \left( \eta_2(x_1,t) - \eta_2(z(t),t) \right)   ,   \n
\end{align}
where once again $\alpha_1(t) \in [-\epsilon,1+ c_9(T-t)]$ and $\varepsilon_1(t) \in [- \epsilon , \epsilon ]$.
For some $\theta(t)\in (0,1]$, 
$$\left| \eta_2(x_0,t) - \eta_2(z(t),t) \right| = \theta(t)  \left| \eta_2(x_1,t) - \eta_2(z(t),t) \right|\,. $$  
Hence, by Claim 1,
\begin{align} 
&\umo(\eta(x_0,t),t)-\umo(\eta(x_1,t),t)\n \\
&\ \
=
\frac{\theta(t) b(t) (|\dn_1(t)| + \delta t) \varepsilon_1(t)}{T-t}\dn_1(t) -\frac{ \varrho(t)\alpha _1(t)}{T-t} \dn_1(t) -   \frac{b(t)(|\dn_1(t)| + \delta t)\nu(t)\alpha_1(t)}{T-t}  \dn_1(t)   .\n
\end{align}
With
$$
\beta_1(t) = [ \varrho(t) + b(t)(|\dn_1(t)| + \delta t)\nu(t)] \alpha _1(t) - \theta(t) b(t) (|\dn_1(t)| + \delta t ) \varepsilon_1(t)\,,
$$
then $\beta_1(t) \in [ -2 \epsilon , 1+2c_9(T-t) ]$ and
$$
\umo(\eta(x_0,t),t)-\umo(\eta(x_1,t),t) = -\frac{ \beta _1(t)}{T-t} \dn_1(t) \,.
$$

Similarly, for $u_2^-$, we have that
 \begin{align}
&\umt(\eta(x_0,t),t)-\umt(\eta(x_1,t),t) \n \\
& \qquad =\left[\umt(\eta(x_0,t),t)-\umt(\eta(\oz (t),t),t)\right]   +\left[\umt(\eta(\oz (t),t),t)-\umt(\eta(x_1,t),t)\right]\n\\
&\qquad  =\frac{\alpha_2(t)}{T-t} \left( \eta_2(x_0,t) - \eta_2(z(t),t) \right)+ \frac{\varrho(t)\varepsilon _2(t)}{T-t} \dn_1(t) +  \frac{ \nu(t)\alpha _2(t)}{T-t}  \left( \eta_2(x_1,t) - \eta_2(z(t),t) \right)\,, 
\n
\end{align}
with $\varepsilon_2(t)\in [-\epsilon,\epsilon]$ and $ \alpha _2(t) \in [- \epsilon , 1+ c_9(T-t) ]$.   Hence, from Claim 1, we see that
\begin{align*} 
&\umt(\eta(x_0,t),t)-\umt(\eta(x_1,t),t)\\
&\qquad
=\frac{\theta(t)b(t)(|\dn_1(t)| + \delta t)\alpha_2(t)}{T-t} \dn_1(t) + \frac{\varrho(t)\varepsilon _2(t)}{T-t} \dn_1(t) +  \frac{\nu(t) b(t) (|\dn_1(t)| + \delta t) \alpha _2(t)}{T-t}  \dn_1(t)\,.
\end{align*} 
Setting
$$
\mathcal{E} _2(t) =[ \theta(t) + \nu(t)] b(t)(|\dn_1(t)| + \delta t)\alpha_2(t)+ \varrho(t)\varepsilon _2(t) \,,
$$
we see that by Claim 2, $ \mathcal{E} _2(t) \in [ -2 \epsilon , 2 \epsilon ]$, and
$$
\umt(\eta(x_0,t),t)-\umt(\eta(x_1,t),t) = \frac{\mathcal{E} _2(t) }{T-t} \dn_1(t) \,.
$$
In this case, $ \dn_t = \M \, \dn$ with 
 \begin{equation*}
\M(t)=\frac{1}{T-t} \left[\begin{array}{cc} - \beta_1(t) & 0\\
                                         \mathcal{E} _2(t) &  0
\end{array}
\right]
\,.
\end{equation*}
which is a special case of the matrix given (\ref{c0})  with $ \varepsilon _1(t)=0$ and $ \alpha _2(t)=0$.  This completes the proof.
\end{proof}

\section{Proof of the Main Theorem} \label{sec_proof}

We now give a proof of Theorem \ref{thm1}.
 We assume that either a splash or splat singularity does indeed occur, and then show
that this leads to a contradiction.

We begin the proof with the case that 
a single splash singularity occurs at time $t=T$  and  that  there exist two points
$x_0$ and $x_1$ in $\Gamma$, such that
$\eta(x_0,T)=\eta(x_1,T)$, as we assumed in Section \ref{sec_geometry}.  (In Sections \ref{multiple_splash} and \ref{splat_section}, we will also rule-out the case of multiple simultaneous  splash singularities,  as well as the splat singularity.)

\subsection{A single splash singularity cannot occur in finite time}\label{sec_singlesplash}
As we stated above, 
for $T-t_0$  sufficiently small and in a small neighborhood of $\eta(x_0,T)$,  the interface $\Gamma(t)$, $t \in [t_0,T)$, 
consists of two curves $\Gamma_0(t)$ and $\mathfrak{r}_1(t)$ evolving towards one another, with $\eta(x_0,t) \in \Gamma_0(t)$ and $\eta(x_1,t) \in \mathfrak{r}_1(t)$.   
We consider the two cases that either $ |\nabla \um(\cdot ,t)|$ remains bounded or blows-up as $t \to T$.

\subsubsection{The case that $ |\nabla \um(\eta(x_0,t),t)| \to \infty $ as $t \to T$}

  We prove that both $ \duo(T) \neq 0 $ and $ \duo(T)  =0$, where recall that $\du(t)$ is given by (\ref{aom2}).

\vspace{.05 in}
 
 \noindent
{\it Step 1.  $\duo\ne 0$ at the assumed splash singularity $\eta(x_0,T)$.}

The scalar product of (\ref{c0}) with $ \dn(t)$ yields
\begin{equation}
\label{718.6}
\p_t|\dn|^2=-2 \frac{\beta_1(t)}{T-t} |\dn_1|^2+2 \frac{\varepsilon_1(t)+\mathcal{E}_2(t)}{T-t} \dn_1\ \dn_2 +2 \frac{\alpha_2(t)}{T-t} |\dn_2|^2\,,
\end{equation}
where the constants $\beta_1(t), \alpha _2(t), \varepsilon_1(t), \mathcal{E}_2(t)$ are defined in Theorem \ref{thm_geom2}.
Therefore, since $T-t <  \epsilon \ll 1$, 
\begin{equation*}
\p_t|\dn|^2\ge -\frac{2+C\epsilon}{T-t} |\dn|^2\,,
\end{equation*}
from which we infer that
\begin{equation}
\label{718.7}
|\dn(t)|^2\ge |\dn(0)|^2 \frac{(T-t)^{2+C\epsilon}}{T^{2+C\epsilon}}\,.
\end{equation}

We now assume that 
\begin{equation}
\label{rev9}
\duo(T) =0\,,
\end{equation}
and now proceed to infer a contradiction from this assumption. Since $\dn(T)=0$ (since we have assumed that a splash singularity occurs at
$t=T$), we have that
\begin{align}
\label{rev10}
\dn_1(t)&=\int_T^t  \left(\p_t \eta(x_0,s) - \p_t \eta(x_1,s) \right)ds \nonumber\\
&=\int_T^t  (v^-_1(x_0,s)-v^-_1(x_1,s)) ds \nonumber\\
&=\int_T^t  (v^-_1(x_0,s)-v^+_1(x_0,s)) ds + \int_T^t  (v^+_1(x_0,s)-v^+_1(x_1,s)) ds - \int_T^t  (v^-_1(x_1,s)-v^+_1(x_1,s)) ds\nonumber\\
&= -\int_T^t  \delta v_1(x_0,s) ds  + \int_T^t  (v^+_1(x_0,s)-v^+_1(x_1,s)) ds+ \int_T^t  \delta v_1(x_1,s) ds\nonumber\\
&= -\int_T^t  \delta v\cdot(\eo-\tau)(x_0,s) ds- \int_T^t  \delta v\cdot\tau (x_0,s) ds + \int_T^t  (v^+_1(x_0,s)-v^+_1(x_1,s)) ds\nonumber\\
&\ \ \ + \int_T^t  \delta v\cdot(\eo-\tau)(x_1,s) ds+\int_T^t  \delta v\cdot\tau(x_1,s) ds\nonumber\\
&= -\int_T^t  \delta v\cdot(\eo-\tau)(x_0,s) ds- \int_T^t  \left[\delta v\cdot\tau (x_0,T)+\int_T^s \p_t(\delta v\cdot\tau)(x_0,l)dl \right] ds \nonumber\\
&\ \ \ + \int_T^t  \left[(v^+_1(x_0,T)-v^+_1(x_1,T))+\int_T^s \p_t (v^+_1(x_0,l)-v^+_1(x_1,l)) dl\right] ds\nonumber\\
&\ \ \ + \int_T^t  \delta v\cdot(\eo-\tau)(x_1,s) ds+\int_T^t  \left[\delta v\cdot\tau(x_1,T)+\int_T^s \p_t(\delta v\cdot\tau)(x_1,l) dl\right] ds
\,.
\end{align}
Using the fact that $\tau(x_0,T)=\eo=\tau(x_1,T)$, (\ref{rev10}) then becomes
 \begin{align}
\label{rev11}
\dn_1(t)
&= -\int_T^t  \delta v\cdot(\eo -\tau)(x_0,s) ds- \int_T^t  \int_T^s \p_t (\delta v\cdot\tau)(x_0,l)dl ds \nonumber\\
&+ \int_T^t  \left[-\delta v_1(x_0,T)+v^+_1(x_0,T)-v^+_1(x_1,T)+\delta v_1(x_1,T)+\int_T^s \p_t(v^+_1(x_0,l)-v^+_1(x_1,l)) dl \right] ds\nonumber\\
& + \int_T^t  \delta v\cdot(\eo-\tau)(x_1,s) ds+\int_T^t  \int_T^s \p_t (\delta v\cdot\tau)(x_1,l) dl ds
\,.
\end{align}
Next, since $-\delta v_1(x_0,T)+v^+_1(x_0,T)-v^+_1(x_1,T)+\delta v_1(x_1,T)=\delta u_1^-(T)$, (\ref{rev11}) and the assumption (\ref{rev9}) then
 provide us with
 \begin{align}
\label{rev12}
\dn_1(t)
&= -\int_T^t  \delta v\cdot(\eo -\tau)(x_0,s) ds- \int_T^t  \int_T^s \p_t(\delta v\cdot\tau) (x_0,l)dl ds \nonumber\\
&\ \ \ + \int_T^t \int_T^s  \p_t(v^+_1(x_0,l)-v^+_1(x_1,l)) dl  ds\nonumber\\
&\ \ \ + \int_T^t  \delta v\cdot(\eo -\tau)(x_1,s) ds+\int_T^t  \int_T^s \p_t(\delta v\cdot\tau)(x_1,l) dl ds
\,.
\end{align}
Due to the $L^\infty$ control of $(\delta v\cdot\tau)_t$ provided by (\ref{jump_v}), and by writing $\eo -\tau(x_i,s)=\int_s^T \tau_t (x_i,s)$ (for $i=0,1$), (\ref{rev12}) allows us to conclude that
\begin{equation}
\label{rev13.0}
|\dn_1(t)|\lesssim \mathcal{M} (T-t)^2\,.
\end{equation}
\textcolor{black}{Note here that we used $\tau_t(x,t)=\tt_t(\eta(x,t),t)+\nabla_{\ttt}\tt (\eta(x,t),t) |\eta'(x,t)|$. By noticing that $\tt_t$ can be computed from Remark 3, we then have $|\tau_t(\cdot,t)|_{L^\infty(\Gamma)}\lesssim\mathcal{M}$.}

Therefore, 
$ |\dn_1(t)|^2 \lesssim \mathcal{M} (T-t)^4$, and from 
 (\ref{718.7}), this implies (since $ 0 \le \epsilon \ll 1$) that
$(T-t)^{2+C\epsilon}  \le C  |\dn_2(t)|^2$; 
hence, by choosing  $t_1\in(0,T)$ sufficiently close to $T$,  for any $t\in [t_1,T]$, 
\begin{align} 
|\dn_1(t)| &\lesssim \mathcal{M} (T-t)^2
= \mathcal{M} (T-t)^{1+C \epsilon /2} (T-t)^{1-C \epsilon /2} \nonumber \\
& \lesssim  (T-t)^{1-C \epsilon /2}\ |\dn_2(t)|  \le  \sqrt{T-t}\ |\dn_2(t)|\,.
\label{rev13}
\end{align}

Using (\ref{rev13}) in (\ref{718.6}) and the fact that $0 < T-t < \epsilon \ll 1$,  we then obtain
\begin{equation}
\label{rev14}
\p_t|\dn|^2\ge -[2(2c_9+\frac{1}{\sqrt{T-t}})\textcolor{black}{+\frac{8\epsilon}{\sqrt{T-t}}+\frac{4\epsilon}{T-t}}]|\dn_2|^2\ge -[2(2c_9+\frac{1}{\sqrt{T-t}})\textcolor{black}{+\frac{8\epsilon}{\sqrt{T-t}}+\frac{4\epsilon}{T-t}}]|\dn|^2\,.
\end{equation}
Thus,
\begin{equation*}
\p_t\left(|\dn|^2 e^{\int_{t_1}^t [2(2c_9+\frac{1+4\epsilon}{\sqrt{T-s}})]\ ds} \textcolor{black}{(T-t)^{-4\epsilon}}\right)\ge 0 \,.
\end{equation*}
Hence,
\begin{equation*}
|\dn|^2(t)\ge |\dn|^2(t_1) e^{-\int_{t_1}^T [2(2c_9+\frac{1+4\epsilon}{\sqrt{T-s}})]\ ds} \textcolor{black}{\frac{(T-t)^{4\epsilon}}{(T-t_1)^{4\epsilon}}> C (T-t)^{4\epsilon}}\,,
\end{equation*}
with $C>0$ finite, since $(T-s)^{-\frac{1}{2}}$ is integrable. This is then in contradiction with our assumption of a splash singularity 
occurring at time $t=T$ \textcolor{black}{which implies  that$$|\delta\eta|^2(t)\lesssim\mathcal{M} (T-t)^2;$$}
therefore,
the assumption (\ref{rev9}) was wrong as it lead to a contradiction, leading us to conclude that, in fact, 
\begin{equation}\label{zzt300}
|\duo(T) | >0 \,.
\end{equation}

\vspace{.1in}

\noindent
{\it Step 2.  $\duo= 0$ at the assumed splash singularity $\eta(x_0,T)$.}  
Having shown that $\duo \neq 0$ at the splash singularity, in order to arrive at a contradiction, we shall next prove that we also have $\duo =0$ at the splash singularity.

\begin{figure}[h]
\begin{center}
\begin{tikzpicture}[scale=0.5]
\draw[color=blue,ultra thick] plot[smooth,tension=.6] coordinates{(-6,5)  (-4,2.0)  (-3, 3) (-2,1.2)  (0,3.6)
(1,5)  };
 \draw[ultra thick, red] (-2,0) sin (0,2);    
 \draw[ultra thick, red] (0,2) cos (2,0);    
\draw (-4,2) node { $\newmoon$}; 
\draw (0,2) node { $\newmoon$}; 
\draw (0,3.6) node { $\newmoon$}; 
\draw(0,2) -- (0,3.6);
\draw[->] (4,2.9) -- (0,2.8);
\draw (4.8,3) node { $\mathfrak{r}_1(t)$};
\draw (-1.8,2.35) node { $\mathfrak{r}_2(t)$};

\draw (-6.,2) node { $\eta(x_0,t)$}; 
\draw (1.8,2.2) node { $\eta(x_1,t)$};
\draw (1.7,3.8) node { $\eta(z(t),t)$};
\draw (-6.8,5.) node { $\Gamma_0(t)$};
\draw (3,-0.1) node { $\Gamma_1(t)$};
\end{tikzpicture} 
\end{center}
\caption{  The portion of the interface $\Gamma_0(t)$,  near $\eta(x_0,t)$,  is
shown to have an oscillation that may only disappear in the limit as $t \to T$.}
\label{fig5}
\end{figure}
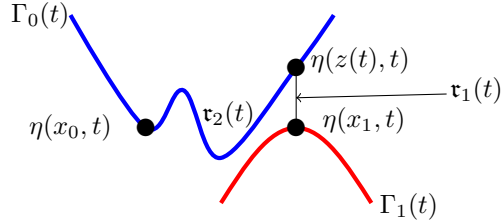

We now define the following two curves. The first curve $\mathfrak{r}_1(t)$ is the vertical segment joining $\eta(x_1,t)\in \mathfrak{r}_1(t)$ to a point $\eta(z(t),t)\in
\Gamma_1(t)$.  This segment is contained in full in the closure of $\Omega^-(t)$ (for $T-t$ sufficiently small), as we have shown 
in Step 1 of the proof of Theorem \ref{thm_geom2}, by simply switching the role of $x_0$ and $x_1$ in the definition of this vertical segment.

The second curve $\mathfrak{r}_2(t)$ is the portion of $\Gamma_0(t)$ linking $\eta(z(t),t)$ to $\eta(x_0,t)$.

We now simply write 
\begin{align}
\duo(t)&=u_1^-(\eta(x_0,t),t)-u_1^-(\eta(z(t),t),t)+u_1^-(\eta(z(t),t),t)-u_1^-(\eta(x_1,t),t)\n\\
&=u_1^-(\eta(x_0,t),t)-u_1^-(\eta(z(t),t),t)+\int_{\mathfrak{r}_1(t)}  \nabla u_1^-\cdot\tau\ dl\n\\
&=u_1^-(\eta(x_0,t),t)-u_1^-(\eta(z(t),t),t)+\int_{\mathfrak{r}_1(t)}  \frac{\p u^-_1}{\p x_2}\ dx_2\,,
\label{i1}
\end{align}
where we have used that $ \et$ is the tangent vector to 
 $\mathfrak{r}_1(t)$ in the last equality of (\ref{i1}).

Next, we estimate the length of the vertical segment $\mathfrak{r}_1(t)$, by simply noticing that
\begin{align}
|\eta(x_0,t)-\eta(x_1,t)|^2&=|\eta(x_0,t)-\eta(z(t),t)|^2+|\eta(z(t),t)-\eta(x_1,t)|^2\n\\
&\ + 2 |\eta(x_0,t)-\eta(z(t),t)| |\eta(z(t),t)-\eta(x_1,t)|\cos\theta\,,
\label{i2}
\end{align}
where $\theta$ denotes  the angle between the two vectors $\eta(x_0,t)-\eta(z(t),t)$ and $\eta(z(t),t)-\eta(x_1,t)$.  Due to (\ref{eps_def}), 
the direction of the tangent vector $\tt$ on $\eta(\gamma_0( \epsilon) , t)$ in a small neighborhood of $\eta(x_0,t)$ is very close to horizontal; in particular,
$ |\tt(\eta(x,t),t) \cdot \et| < \epsilon $ for all $x \in \gamma_0( \epsilon )$ and $t \in [t_0( \epsilon) , T)$.
Hence, we  have that 
$\eta(x_0,t)-\eta(z(t),t)$ is in direction close to horizontal. On the other hand, $\eta(z(t),t)-\eta(x_1,t)$ is in the vertical direction. 
Therefore, $\theta$ is very close to $\frac{\pi}{2}$ which then, in turn, implies from (\ref{i2}) that
\begin{align*}
|\eta(x_0,t)-\eta(x_1,t)|^2&\ge|\eta(x_0,t)-\eta(z(t),t)|^2+|\eta(z(t),t)-\eta(x_1,t)|^2\n\\
&\qquad -\frac 1 2 |\eta(x_0,t)-\eta(z(t),t)| |\eta(z(t),t)-\eta(x_1,t)|\n\\
&\ge\frac 3 4 |\eta(x_0,t)-\eta(z(t),t)|^2+\frac 3 4|\eta(z(t),t)-\eta(x_1,t)|^2\,,
\end{align*}
which shows that the square of the length of the vertical segment satisfies
\begin{align}
|\eta(x_1,t)-\eta(z(t),t)|^2\le & \frac 4 3 |\eta(x_0,t)-\eta(x_1,t)|^2\n\\
\le & \frac 4 3 |\eta(x_0,t)-\eta(x_0,T)-\eta(x_1,t)+\eta(x_1,T)|^2\n\\
\le & \frac 4 3 \left|\int_T^t v^-(x_0,s)\ ds -\int_T^t v^-(x_1,s)
\ ds\right|^2\n\\
\le & \frac{16}{3} (T-t)^2 \|v^-\|^2_{L^\infty(\Gamma)}\n\\
\lesssim &\mathcal{M}^2 (T-t)^2
\,.
\label{i3}
\end{align}
thanks to Lemma \ref{lemma1}.

Then, with our estimate (\ref{u12epsilon}) on $\frac{\p u^-_2}{\p x_1}$  and the fact that $ \operatorname{curl} u^-=0$, we then have with (\ref{i3}) that
\begin{equation}
\label{i4}
\left|\int_{\mathfrak{r}_1(t)} \nabla u_1^-\cdot\tau\ dl\right|\lesssim {\mathcal{M}}\ (T-t) \frac{\epsilon}{T-t}=\epsilon {\mathcal{M}} \,.
\end{equation}

It remains to estimate the difference $u_1^-(\eta(x_0,t),t)-u_1^-(\eta(z(t),t)$ appearing on the right-hand side of (\ref{i1}).    Recall that  $\Gamma_0(t)=\eta(\gamma_0(\epsilon),t)$, for $\epsilon>0$ small enough fixed.   From Lemma \ref{lemma1}, 
 $v^-$ is continuous along $\Gamma_0$. Next, we have that $\eta$ is continuous and injective from $\overline{\gamma_0(\epsilon)}\times [0,T]$, into its image $\aleph$. Since $\eta$ is continuous and injective, and $\overline{\gamma_0 (\epsilon)}$ is closed, $\aleph$ is closed (as the sequential definition of a closed set is straightforwardly satisfied).
 As a result, $\eta^{-1}$ is also continuous and injective from $\aleph$ into $\overline{\gamma_0(\epsilon)}\times [t_0( \epsilon ),T]$, as the sequential definition of continuity is straightforwardly satisfied. 
 By composition, $u^-=v^-\circ\eta^{-1}$ is also continuous on $\aleph$. Since $z(t)\in\gamma_0(\epsilon)$ by step 1 of the proof of Theorem \ref{thm_geom2} (by switching the roles of $x_0$ and $x_1$), and $z(t)$ converges to $x_0$ as $t\rightarrow T$, we then have that $\eta(z(t),t)$ belongs to $\aleph$ and satisfies $$\lim_{t\rightarrow T} ( \eta(z(t),t)-\eta(x_0,t))=0\,.$$ Since we just established the continuity of $u^-$ on $\aleph$, and henceforth its uniform continuity in the compact set $\aleph$, we can infer from the previous limit and this uniform continuity that $u_1(\eta(x_0,t),t)-u_1(\eta(z(t),t)$ converges to zero as $t\rightarrow T$.

With this fact, we can infer from (\ref{i1}) and (\ref{i3})  that as $t\rightarrow T$
\begin{equation*}
|\duo(T)|\le \epsilon \mathcal{M}\,,
\end{equation*}
this being true for any $\epsilon>0$. Therefore,
\begin{equation}
\n
|\duo(T)|=0\,,
\end{equation}
which is a contradiction with (\ref{zzt300}).

We shall next  explain why a non-singular gradient of the velocity $u^-$ also does not allow for a splash singularity, which will finish the 
proof of our main result in the case of a single self-intersection.

\subsubsection{The case that $|\nabla \um(x,t)|$ remains bounded}  If $\|\nabla \um(\cdot ,t)\|_{L^ \infty (\Om(t))}$ is bounded on $[0,T]$, we can still obtain the differential equation
$\dn_t(t) = \M(t) \dn(t)$  using the same path integral that we used in the proof of Theorem \ref{thm_geom2},  with paths shown in Figure \ref{fig4}; 
in this case, however, the components of the matrix $\M$ are bounded on $[0,T]$.  From $\dn_t(t) = \M(t) \dn(t)$,
we see that
\begin{equation}
\n
\p_t|\dn|^2=2 {\M}_{11} |\dn_1|^2+2\left( {{\M}_{12}(t)+{\M}_{21}(t)} \right)\dn_1\ \dn_2 +2 {\M}_{22} |\dn_2|^2\,.
\end{equation}
with $\M_{ij}$ bounded for $i,j=1,2$.
Therefore,
\begin{equation*}
\p_t|\dn|^2\ge -C(\mathcal{\M}) |\dn|^2\,,
\end{equation*}
which then provides
\begin{equation}
\n
|\dn(t)|^2\ge |\dn(0)|^2 e^{-C(\mathcal{\M})t}\,.
\end{equation}
Since $\dn(0)\ne 0$, we then cannot have $\dn(T)=0$ for any finite $T$.

\subsubsection{The case that the region between $x_0$ and $x_1$ is $\Omega^+$:}
 In this case, we can still proceed with the same geometric constructions as before. The difference is that in this case, the matrix 
 $\M(t)$ has bounded coefficients (since $\nabla u^+$ is bounded in $L^\infty(\Omega^+(t)$), and therefore, we are in the same
  situation as the case treated previously where $|\nabla \um(x,t)|$ remains bounded, which leads to the impossibility of a splash
   singularity at time $T$.

\subsection{An arbitrary number (finite or infinite) of splash singularities at time $T$ is not possible}\label{multiple_splash}
We assume that an arbitrary number of simultaneous splash singularities occur at time $T>0$.   We now  focus on one of the many possible self-intersection points.
To this end,  let $x_0$ and $x_1$ be two points in $\Gamma$ such that $\eta(x_0,T)=\eta(x_1,T)$.   Let $\Gamma_0 \subset \Gamma$ be a local neighborhood
of $x_0$ and let $\Gamma_1 \subset \Gamma$ be a local neighborhood of $x_1$.

Then, there exists a sequence of points $x_0^n\in\Gamma_0$ converging to $x_0$, 
and of a sequence of points $x_1^n\in\Gamma_1$ converging to $x_1$  such that 
\begin{equation}
\label{mns1}
d_0^n:= d(\eta(x_0^n,T),\eta(\Gamma_1,T))\ne 0\,,\ \ \ \  d_1^n:= d(\eta(x_1^n,T),\eta(\Gamma_0,T))\ne 0 \ \ \forall n \in \mathbb{N}  \,,
\end{equation}
where $d$ denotes the distance function;  otherwise, if (\ref{mns1}) did not hold, then we would have non trivial neighborhoods $\gamma_0$ of $x_0$ and 
$\gamma_1$ of $x_1$ such that $\eta(\gamma_0,T)=\eta(\gamma_1,T)$, which means a splat singularity occurs at $t=T$,  and we treat that case below in Section \ref{splat_section}.

We continue to let $\eo$ denote a tangent unit vector to $\Gamma(T)$  at the splash contact point $\eta(x_0,T)$.  We then have, by the continuity of the tangent vector $\tt$ to the interface,  
that for  both  sequences of points, 
\begin{equation}
\label{mns2}
\left|\eo-\tt(\eta(x_0^n,T),T)\right|\le \epsilon\,,
\end{equation}
for $\epsilon>0$ fixed and $n$ large enough. 
We now call $z_1^n$ the orthogonal projection of $\eta(x_0^n,T)$ onto $\eta(\Gamma_1,T)$. We then have from (\ref{mns1}) that
\begin{equation}
\label{mns3}
\left|\eta(x_0^n,T)-z_1^n\right|=d_0^n>0\,.
\end{equation}
Furthermore, we denote by the unit vector $e_0^n$ the direction of the vector $\eta(x_0^n,T)-z_1^n$ 
(with base point at $z_1^n$ and ``arrow'' at $\eta(x_0^n,T)$).
By definition, $e_0^n$ points in the normal direction  to $\eta(\Gamma_1,T)$ at $z_1^n$ and by (\ref{mns2}), $e_0^n$ is close to $\et$.
For each point $x_0^n$, the segment  $(\eta(x_0^n,T),z_1^n)$ is contained in $\eta(\Omega^-,T)$.

By continuity of $\eta$ on $\Gamma\times [0,T]$ we also infer from (\ref{mns3}) that there exists a connected neighborhood 
$\gamma_0^n$ of $x_0^n$ on $\Gamma$, of length $L_n>0$, such that for any $x\in\gamma_0^n$ we have
\begin{equation}
\label{mns4}
 d(\eta(x,T),\eta(\Gamma_1,T))\ge \frac{d_0^n}{2};
 \end{equation}
moreover, the direction of the vector $\eta(x,T)-P_{\eta(\Gamma_1,T)}(\eta(x,T))$,  normal to $\eta(\Gamma_1,T)$ at $P_{\eta(\Gamma_1,T)}(\eta(x,T))$,  is 
 close to $\et$, where $P_{\eta(\Gamma_1,T)}$ denotes the orthogonal projection onto $\eta(\Gamma_1,T)$.

Note that for each $x\in\gamma_0^n$, the segment $(\eta(x,T),P_{\eta(\Gamma_1,T)}(\eta(x,T)))$ is contained in $\eta(\Omega^-,T)$.
By continuity of the direction of these vectors, we then have that 
\begin{equation}
\label{mns5}
\omega_n=\cup_{x\in\gamma_0^n}(\eta(x,T),P_{\eta(\Gamma_1,T)}(\eta(x,T)))\,,
\end{equation} 
is an open set contained in $\eta(\Omega^-,T)$.
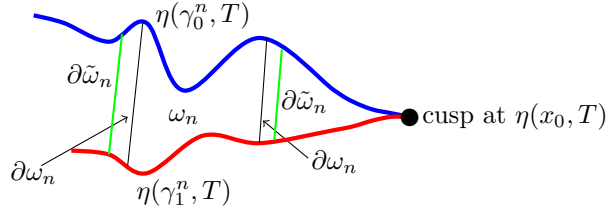
\begin{figure}[ht]
\begin{center}
\begin{tikzpicture}[scale=0.5]

\draw (-1.5,3.2) node { $\eta(\gamma_0^n, T)$}; 
\draw (-2,-1.5) node { $\eta(\gamma_1^n, T)$}; 
\draw (-4.5,1.6) node { $\partial \tilde \omega_n$}; 
\draw (-6,-1) node { $\partial  \omega_n$}; 
\draw[->]  (-5.8,-.8) -- (-3.5,.5); 
\draw (2,-.7) node { $\partial  \omega_n$}; 
\draw[->]  (1.3,-.6) -- (.1,.4); 
\draw (1.2,.9) node { $\partial \tilde \omega_n$}; 
\draw (-2,.5) node { $\omega_n$}; 
\draw[color=blue,ultra thick] plot[smooth,tension=.6] coordinates{(-6,3.2) (-5.2,3)  (-4,2.5)  (-3, 3) (-2,1.2)  (0,2.6)   (2.0, 1.2) 
                                                                                                                     (2.5, .9) (3, .7) (3.5, .6) (4, .5)};
\draw[color=red,ultra thick] plot[smooth,tension=.6] coordinates{(-5, -.4) (-4,-.5)  (-3,-1)  (-1.5,0)    (0,-.2)  (2., .15)  (2.5,.25)  (3,.4) (3.5, .5) (4,.5)};
 \draw (4,.5) node { $\newmoon$}; 
\draw (6.8,.5) node { cusp at $\eta(x_0,T)$}; 

\draw[color=green, thick] plot[smooth,tension=.6] coordinates{  (-3.65,2.7)  (-4,-.5) };

\draw[color=green, thick] plot[smooth,tension=.6] coordinates{  (.6, 2.3)  (.4, -.1)  };

\draw  (-3.1,3) -- (-3.5,-.8); 
\draw  (0.2,2.6) -- (0,-.2); 

\end{tikzpicture} 
\end{center}
\caption{The open set $\omega_n$ is contained in the larger  open set $\tilde \omega_n$}\label{figwn}
\end{figure}
Furthermore,  $\partial \omega_n$ contains the set $ \eta(\gamma_0^n,T)$ of length $L_n>0$ (as its top boundary), and by continuity of the directions, $\partial
\omega_n$ also contain a connected subset
 $\eta(\gamma_1^n,T)$ of $\eta(\Gamma_1,T)$, of length greater than $\frac{L_n}{2}$ (as its bottom boundary).   
 Because $\omega_n$ does not intersect the cusp which occurs at the 
 contact point, we  define the open set   $ \tilde \omega_n\supset \omega_n$, such that  the lateral part of $\partial \tilde \omega_n$ is parallel to
 the lateral part of $\partial \omega _n$ and connects 
 $ \eta(\Gamma_0,T)$ and $\eta(\Gamma_1,T)$ as shown in Figure \ref{figwn}.

Next, we introduce the stream functions $\psi^\pm$ such that $u^\pm(\cdot ,T)=\nabla^\perp\psi^\pm$, and we recall that $u^+$ (and hence $\psi^+$) has the good
regularity on $\Gamma(t)$ for $t \in [0,T]$, given by (\ref{bounds}).   Let $\mathscr{W}_n$ be an open set such that $\omega_n \subset \mathscr{W}_n \subset
\tilde \omega _n$.
Let $0\le \vartheta_n \le 1$ denote a $C^ \infty $ cut-off function which is equal to $1$ in
$\overline{\omega_n}$ and equal to $0$ on $\overline{\tilde \omega _n}/ \mathscr{W}_n$.

We have that $\psi^-$  is an $H^1(\Omega^-(T))$ weak solution of $ \Delta \psi^- =0$ in $\Omega^-(T)$ and $\psi^- =\psi^+$ on $\partial \Omega^-(T)$.
Then
$\vartheta_n\psi^-$  satisfies
\begin{alignat*}{2}
-\Delta( \vartheta_n \psi^-)&=-\psi^- \Delta \vartheta_n -2 \nabla \vartheta_n \cdot \nabla \psi^-\,,\ \ &&  \text{ in }\ \tilde \omega_n\,,\\
\vartheta_n \psi^- &=\psi^+ && \text{ on }\ \eta(\Gamma_0,T)\cup\eta(\Gamma_1,T)\cap \partial \tilde \omega_n\,,
\end{alignat*}
and as $\psi^+\in H^{3.5}(\eta(\Gamma_0,T))\cup H^{3.5}(\eta(\Gamma_1,T))$,  standard elliptic regularity shows that 
\begin{equation*}
\psi^- \in H^{4} (\omega_n)\,,
\end{equation*}
and therefore that
\begin{equation}
\label{mns6}
\nabla u^-(\cdot,T) \in H^{3} (\omega_n)\subset L^\infty(\omega_n)\,.
\end{equation}

\def\Dnr{ \mathscr{D}_n^r }
\def\Dnl{ \mathscr{D}_n^l}

Let $\Dnr$ denote the pre-image of $\omega_n$ under the map $\eta(\cdot,T)$. Let us assume that $\partial \Dnr \cap \Gamma_0$ lies to the right of $x_0$.
Since $\overline{\omega_n}$ does not intersect  the splash singularity at time $T$, 
$\eta(\cdot,T)$ is bijective and continuous from $\Dnr$ into $\omega_n$, and therefore $\Dnr$ is an open connected set.

Furthermore, $\nabla u^-\circ\eta$ is  also continuous on $\overline{\Dnr}\times [0,T]$ which, thanks to (\ref{mns6}), shows that for all $t\in [0,T]$,
\begin{equation}
\label{mns7}
\|\nabla u^-(\cdot,t)\|_{L^\infty(\eta(\Dnr,t))}\le \mathcal{M}^r_n\,.
\end{equation}

 We can also choose the sequence $x_0^n$ to lie on the left of $x_0$ (otherwise, we would have a splat singularity). This similarly gives an open neighborhood 
 $\Dnl$ of the same type as $\Dnr$ satisfying 
for all $t\in [0,T]$,
\begin{equation}
\label{mns8}
\|\nabla u^-(\cdot,t)\|_{L^\infty(\eta(\Dnl,t))}\le \mathcal{M}^l_n\,.
\end{equation}

\def\Cnr{\mathscr{C}_n^r}
\def\Cnl{\mathscr{C}_n^l}
 
We now denote by $\Cnr$ (respectively $\Cnl$)  the lateral part of $\partial\Dnr$ (respectively $\partial \Dnl$) joining $\Gamma_0$ to $\Gamma_1$, 
and we denote by $\aleph_n$ the open set delimited by 
$\Cnr$;  the subset of $\Gamma_0$ containing $x_0$ linking $\Cnr$ to $\Cnl$;  $\Cnl$; and the subset of $\Gamma_1$ containing $x_1$ linking $\Cnl$ to $\Cnr$.

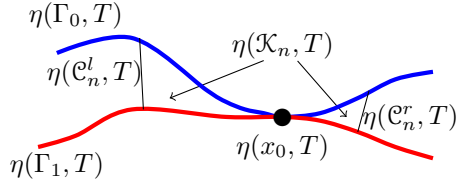
\begin{figure}[ht]
\begin{center}
\begin{tikzpicture}[scale=0.5]

\draw (-1.5,3.2) node { $\eta(\Gamma_0, T)$}; 
\draw (-2,-0.8) node { $\eta(\Gamma_1, T)$}; 
\draw (4,2.4) node { $\eta(\mathscr{K}_n,T)$}; 
\draw[->]  (3.5,2) -- (1,1); 
\draw[->]  (4.4,1.99) -- (5.7,.5); 
\draw (7.4,.5) node { $\eta(\mathscr{C}^r_n,T)$}; 
\draw (-.99,1.7) node { $\eta(\mathscr{C}^l_n,T)$}; 



\draw[color=blue,ultra thick] plot[smooth,tension=.6] coordinates{ (-2,2.2)  (0,2.6)   (2.0, 1.2) 
                                                                                                                     (2.5, .9) (3, .7) (3.5, .6) (4, .5) (5, .6) (6, 1) (7, 1.5) (8, 1.7)};
\draw[color=red,ultra thick] plot[smooth,tension=.6] coordinates{ (-2.5, -.3) (-1.5,0)    (0,.7)   (2.5,.5)  (4,.5) (5, .4) (6, .1) (7, -.3) (8, -.6)};
 \draw (4,.5) node { $\newmoon$}; 
\draw (4,-.3) node {$\eta(x_0,T)$};

\draw  (0.2,2.6) -- (.3,.7); 
\draw  (6,.1) -- (6.3,1.2); 

\end{tikzpicture} 
\end{center}
\caption{The region in which we apply the maximum and minimum principle.}\label{figwn0}
\end{figure}

For $n$ large enough, we will have estimate (\ref{mns2}) satisfied at any point of $\partial\aleph_n\cap\Gamma$, with moreover the length of $\partial\aleph_n\cap\Gamma$ being of order $\epsilon$. This then implies in a way similar to Step 4 of Theorem \ref{thm_blowup}, that
\begin{equation}
\label{mns9}
\left\|\frac{\partial u_2^-}{\partial x_1}(\cdot,t)\right\|_{L^\infty(\eta(\partial\aleph_n\cap\Gamma,t))}\le \frac{\epsilon}{T-t}\,,
\end{equation}
for any $t<T$. Moreover, for $t$ close enough to $T$, the maximum of the two constants $\mathcal{M}_n^r$ and $\mathcal{M}_n^l$  of (\ref{mns7}) and (\ref{mns8}) will become smaller than $\frac{\epsilon}{T-t}$.  Thus, for any $t<T$ close enough to $T$,
\begin{equation*}
\left\|\frac{\partial u_2^-}{\partial x_1}(\cdot,t)\right\|_{L^\infty(\eta(\partial\aleph_n,t))}\le \frac{\epsilon}{T-t}\,,
\end{equation*}
which by application (for each fixed $t<T$ close enough to $T$) of the maximum and minimum principle for the harmonic function $\frac{\partial u_2^-}{\partial x_1}(\cdot,t)$ on the open set $\eta(\aleph_n,t)$ provides
\begin{equation}
\label{mns10}
\left\|\frac{\partial u_2^-}{\partial x_1}(\cdot,t)\right\|_{L^\infty(\eta(\aleph_n,t))}\le \frac{\epsilon}{T-t}\,.
\end{equation}
 We can then apply the same arguments as in the Sections \ref{sec_geometry} and \ref{sec_singlesplash} to exclude a splash singularity associated with $x_0$ and $x_1$ simply by working in the neighborhood of size $C \epsilon$ ($C$ bounded from below away from $0$) where (\ref{mns10}) holds.

\subsection{A splat singularity is not possible}\label{splat_section}
We now assume the existence of a splat singularity: 
there exists two disjoint closed subsets of $\Gamma$, which we denote by $\Gamma_0$ and $\Gamma_1$, with non-zero measure, 
such that contact occurs at time $t=T$ and $\eta(\Gamma_0,T)=\eta(\Gamma_1,T)$. We furthermore assume that the set 
\begin{equation}
\label{nsplat1}
\mathcal{S}_0=\{x\in \Gamma_0\ : \lim_{t\rightarrow T} |\nabla u^-(\eta(x,t),t)|=\infty\}\,,
\end{equation}
has a non-empty interior, and denote by $x_0$ and $y_0$ two distinct points on $\mathcal{S}_0$ such that the curve $\gamma_0\subset\Gamma_0$, which connects the points $x_0$ to $y_0$, is contained in $\mathcal{S}_0$. 
We denote by $L(t)$ the length of the curve $\eta(\gamma_0,t)$, which is given by
\begin{equation}
\label{nsplat2}
L(t)=\int_{\gamma_0} |\eta'(x,t)|\ dl\,.
\end{equation}
By Lemma \ref{lemma2}, for any $x\in\mathcal{S}_0$,  $\lim_{t\rightarrow T} \eta'(x,t)=0$, and from Lemma \ref{lemma1}, we have the uniform bound $\sup_{t \in [0,T]} |\eta'|_{L^\infty(\Gamma)}\le \mathcal{M}$ where 
$\mathcal{M}$ is independent of $t<T$. Therefore, by the  dominated convergence theorem,
\begin{equation}
\label{nsplat3}
\lim_{t\rightarrow T} L(t)=0\,,
\end{equation}
which shows that $\eta(x_0,T)=\eta(y_0,T)$, which is a contradiction with the fact that $\eta$ is injective on $\Gamma_0\times [0,T]$. Therefore our assumption that $\mathcal{S}_0$ has non-empty interior was wrong, which shows that this set has empty interior. Therefore the set
\begin{equation}
\label{nsplat4}
\mathcal{B}_0=\{x\in \Gamma_0\ : \lim_{t\rightarrow T} |\nabla u^-(\eta(x,t),t)|<\infty\}\,,
\end{equation}
is dense in $\Gamma_0$. Furthermore, by Lemma
\ref{lemma1},  $|v'( \cdot ,t)|_{L^\infty(\Gamma)}\le \mathcal{M}$ where $\mathcal{M}$ is independent of $t<T$. Hence, by Lemma \ref{lemma2},
$\mathcal{B} _0$   is defined equivalently as
\begin{equation*}
\mathcal{B}_0=\{x\in \Gamma_0\ : |\eta'(x,T)|>0\}\,,
\end{equation*}
which shows that this set is open in $\Gamma_0$. Therefore, $\mathcal{B}_0$ is an open and dense subset of $\Gamma_0$.

Now since $\eta$ is continuous and injective from $\Gamma_0\times [0,T]$ onto its image, it also is a homeomorphism from $\Gamma_0\times [0,T]$ onto its image, which shows that $\eta(\mathcal{B}_0,T)$ is open and dense in 
$\eta(\Gamma_0,T)$.  With
\begin{equation}
\label{nsplat5}
\mathcal{B}_1=\{x\in \Gamma_1\ : \lim_{t\rightarrow T} |\nabla u^-(\eta(x,t),t)|<\infty\}\,,
\end{equation}
 the same argument shows that  $\eta(\mathcal{B}_1,T)$
is also open and dense in $\eta(\Gamma_1,T)$.  Our assumption of a splat singularity at $t=T$  means that 
$\eta(\Gamma_0,T)=\eta(\Gamma_1,T)$, showing that 
$\eta(\mathcal{B}_0,T)$ and $\eta(\mathcal{B}_1,T)$ are two open and dense sets in $\eta(\Gamma_0,T)=\eta(\Gamma_1,T)$. 
They, therefore, have an open and dense intersection. 

Let $\mathcal{Z} $ be a point in this intersection with tangent direction given by $\eo$.  By definition, there exists $z_0\in\mathcal{B}_0$ and $z_1\in\mathcal{B}_1$ such that $\eta(z_0,T)=\eta(z_1,T)$. 
We are therefore back to the case where  interface self-intersection occurs with non-singular $\nabla u^-$ 
(from the definition of the sets $\mathcal{B}_0$ and $\mathcal{B}_1$), except that we do not have an estimate for $\nabla u^-$ valid for the entire interface 
$\Gamma(t)$.

We now consider two open connected curves $\gamma_0\subset\mathcal{B}_0$ and $\gamma_1\subset\mathcal{B}_1$ such that for any point $z_0\in\gamma_0$ there exist a point $z_1\in\gamma_1$ such that $\eta(z_0,T)=\eta(z_1,T)$. For $t\in [T_0,T)$, $T_0$ being very close to $T$, the two curves $\eta(\gamma_0,t)$ and $\eta(\gamma_1,t)$  are very close to each 
other, and at each point, have tangent vector close to $\eo$ (to ensure this, 
 if necessary,  we take a sufficiently small subset of each of these two curves). 
 
Furthermore, from the definition of $\mathcal{B}_0$, we have that the length of the curve $\eta(\gamma_0,t)$ for $t\in [T_0,T)$, $T_0$ being very close to $T$, is close to a number $L_0>0$ (which is the length of $\eta(\gamma_0,T)=\eta(\gamma_1,T)$). Similarly,
the length of the curve $\eta(\gamma_1,t)$ for $t\in [T_0,T)$,  is close to $L_0$. 

We now fix two distinct and close-by points $\eta(z_0,T_0)$ and $\eta(\tilde z_0,T_0)$ on 
$\eta(\gamma_0,T_0)$ such that $|\eta(z_0,T_0) - \eta(\tilde z_0,T_0)| <\frac{L_0}{200}$,   and the distance between each of these points and 
the complement of $\eta(\gamma_0,T_0)$ in $\eta(\Gamma_0,T_0)$ is greater than $\frac{L_0}{4}$. 
By taking $T_0$ closer to $T$ if necessary, we can assume that for any $t\in[T_0,T]$ the distance between $\eta(z_0,t)$ (or $\eta(\tilde z_0,t)$) and the complement of $\eta(\gamma_0,t)$ in $\eta(\Gamma_0,t)$ is greater than $\frac{L_0}{5}$.

As shown in Figure 8, 
we now define $\eta(z_1,T_0)$ as being the intersection of the vertical line passing through $\eta(z_0,T_0)$ and $\eta(\gamma_1,T_0)$. This defines a unique point since the tangent vector to $\eta(\gamma_1,T_0)$ is close to $\eo$, and furthermore the segment 
$(\eta(z_0,T_0),\eta(z_1,T_0))$ is contained in $\eta(\Omega^-,T_0)$. 
Similarly, we define $\eta(\tilde z_1,T_0)$ as being the intersection of the vertical line passing through $\eta(\tilde z_0,T_0)$ and $\eta(\gamma_1,T_0)$. This defines a unique point, with the segment $(\eta(\tilde z_0,T_0),\eta(\tilde z_1,T_0))$  contained in 
$\eta(\Omega^-,T_0)$.

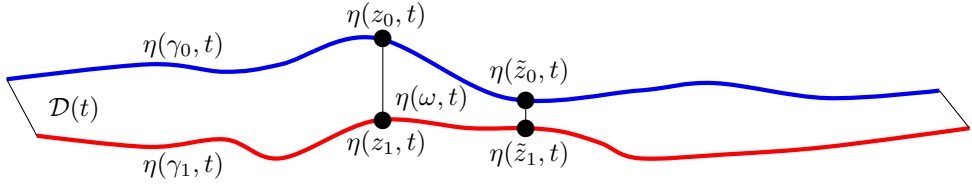
\begin{figure}[ht]
\begin{center}
\begin{tikzpicture}[scale=0.5]

\draw (-5.1,2.7) node { $\eta(\gamma_0, t)$}; 
\draw (-5.1,-0.5) node { $\eta(\gamma_1, t)$};



\draw[color=blue,ultra thick] plot[smooth,tension=.6] coordinates{(-9.8, 1.8) (-6, 2.2) (-4,2) (-2.5,2.2)  (0,2.9)   
                                                                                                                      (3.5, 1.3)  (7, 1.5) (9, 1.7) (12, 1.3) (15, 1.5)};
\draw[color=red,ultra thick] plot[smooth,tension=.6] coordinates{(-9, .3) (-6, 0) (-4, .2) (-2.5, -.3)     (0,.7)   (2.5,.5)  (4,.5) (5, .4) (6, .1) 
                                                                                                                  (7, -.3) (9.3, -.2) (12, 0) (15.8, .5)};

\draw  (15.8,.5) -- (15,1.5); 
\draw  (-9.8,1.8) -- (-9,.3); 
                                                 
\draw (-8,1) node { $\mathcal{D} (t)$};                                                                                                                   
                                                                                                                  
 \draw (.2,.7) node { $\newmoon$}; 
 \draw (.2,2.9) node { $\newmoon$}; 

 \draw (4,1.25) node { $\newmoon$}; 
 \draw (4,.5) node { $\newmoon$}; 

\draw (.3,3.5) node { $\eta(z_0, t)$}; 
\draw (.3,.1) node { $\eta(z_1, t)$}; 

\draw (1.5,1.3) node { $\eta(\omega, t)$}; 

\draw (4.1,2.) node { $\eta(\tilde z_0, t)$}; 
\draw (4.1,-.2) node { $\eta(\tilde z_1, t)$}; 
\draw  (0.2,2.9) -- (.2,.7); 
\draw  (4,1.25) -- (4,.5); 

\end{tikzpicture} 
\end{center}
\caption{That portion of $\Omega^-(t)$ being squeezed together by the approaching splat singularity.}\label{fig8}
\end{figure}

By taking $T_0$ closer to $T$ if necessary, we can assume that for any $t\in[T_0,T]$ the distance between $\eta(z_1,t)$ (or $\eta(\tilde z_1,t)$) and the complement of $\eta(\gamma_1,t)$ in $\eta(\Gamma_1,t)$ is greater than $\frac{L_0}{5}$.
By further taking $T_0$ closer to $T$, if necessary,  we can also assume that
\begin{equation}
\label{250714.1}
\operatorname{dist} (\eta(\gamma_0,T_0),\eta(\gamma_1,T_0))\le \frac{L_0}{100}\,,
\end{equation}
and also that
\begin{equation}
\label{250714.2}
\left(1+ \sup_{[0,T]} \| v^-(  \cdot , t)\|_{L^ \infty ( \Omega ^-)} \right) (T-T_0)< \frac{L_0}{12}\,. 
\end{equation}

We denote by $\eta(\omega,T_0)$ the domain enclosed by the two vertical segments $[\eta(z_0,T_0),\eta(z_1,T_0)]$, $[\eta(\tilde z_0,T_0),\eta(\tilde z_1,T_0)]$, the portion of the curve $\eta(\gamma_0,T_0)$ linking $\eta(z_0,T_0)$ to $\eta(\tilde z_0,T_0)$, and the portion of the curve $\eta(\gamma_1,T_0)$ linking $\eta(z_1,T_0)$ to $\eta(\tilde z_1,T_0)$.
 This domain is contained in $\eta(\Omega^-,T_0)$ (which justifies its name $\eta(\omega,T_0)$, for $\omega\subset\Omega^-$), and has a 
 non-zero area $A_0$ (since its boundary contains two distinct vertical lines and two near horizontal and distinct curves).

By incompressibility, for any $t\in[T_0,T)$, the area of $\eta(\omega,t)$ remains a constant which we call $A_0$. 
Now, as $t\rightarrow T$, the two curves $\eta(\gamma_0,t)$ and $\eta(\gamma_1,t)$ get  close to a splat contact (which occurs at $t=T$);
therefore, the domain $\mathcal{D}(t)$ between these two curves and the two  short lateral segments joining them has an area converging to zero (see Figure \ref{fig8}). Therefore for $t<T$ close enough to $T$ we cannot have $\eta(\omega,t)\subset\mathcal{D}(t)$, as points on the lateral edges of
$\eta(\omega,t)$ would be pushed-out of the lateral boundaries of $ \mathcal{D} (t)$.

Therefore, we have at least a point (in fact a subset of non zero area) $\eta(z,t)$ ($z\in\omega$) such that
\begin{equation}
\label{250714.3}
\left|\eta(z_0,t)-\eta(z,t)\right|\ge \frac{L_0}{5}\,.
\end{equation}
From (\ref{250714.1}), and from the fact that the boundary of $\eta(\omega,T_0)$ is comprised of two vertical segments of length less than $\frac{L_0}{100}$ and of two near horizontal curves of length less than $\frac{L_0}{100}$, we have that
\begin{equation}
\label{250714.4}
\left|\eta(z_0,T_0)-\eta(z,T_0)\right|\le \frac{L_0}{50}\,.
\end{equation}
From (\ref{250714.3}) and (\ref{250714.4}) we then have
\begin{equation}
\label{250714.5}
\left|\int_{T_0}^t [v(z_0,s)-v(z,s)] ds\right|\ge \frac{L_0}{5}-\frac{L_0}{50}\ge \frac{L_0}{6}\,.
\end{equation}
Using (\ref{250714.5}),we infer that
\begin{equation*}
 2(T-T_0) \sup_{[0,T]} \|v\|_{L^\infty(\Omega^-)} \ge \frac{L_0}{6}\,,
\end{equation*}
which is in contradiction with (\ref{250714.2}).    This establishes the impossibility of a splat singularity at time $t=T$. 

As our  analysis  was reduced to a local neighborhood of any assumed splat singularity, as shown in Figure \ref{fig8},  this means that any combination of splat and splash singularities at time $t=T$ can be analyzed in the same way.   This finishes the proof of the exclusion of splat or splash  singularities in finite time.

\section*{Acknowledgments}
DC was supported by the Centre for Analysis and Nonlinear PDEs funded by the UK EPSRC grant EP/E03635X and the Scottish Funding Council.  SS was supported by the National Science Foundation under grants DMS-1001850 and DMS-1301380, by OxPDE at the University of Oxford, 
and by the Royal Society Wolfson Merit Award.   We are grateful to the referees for the time and care in reading our manuscript and for the excellent suggestions that have improved the presentation greatly.

\end{document}